\documentclass{amsart}
\usepackage{amsthm,amssymb}
\usepackage[pagebackref,colorlinks,linkcolor=red,citecolor=blue,urlcolor=blue,hypertexnames=false]{hyperref}
\usepackage{amsrefs}
\usepackage{amscd}
\usepackage{mathrsfs}
\input xypic
\xyoption{all}
\newdir{ >}{{}*!/-5pt/@{>}}
\usepackage{amsmath}
\usepackage{txfonts}
\usepackage{srcltx}
\usepackage{enumitem}

\newtheorem{theorem}[equation]{Theorem}

\newtheorem{definition}[equation]{Definition}

\newtheorem{lemma}[equation]{Lemma}

\newtheorem{proposition}[equation]{Proposition}

\theoremstyle{plain}
\newtheorem{thm}[equation]{Theorem}
\newtheorem{prop}[equation]{Proposition}
\newtheorem{coro}[equation]{Corollary}
\newtheorem{lem}[equation]{Lemma}

\theoremstyle{definition}
\newtheorem{defi}[equation]{Definition}

\theoremstyle{remark}
\newtheorem{nota}[equation]{Notation}
\newtheorem{rem}[equation]{Remark}

\theoremstyle{definition}
\newtheorem{exa}[equation]{Example}

\newtheorem{constru}[equation]{Construction}

\numberwithin{equation}{section}

\newcommand{\comment}[1]{} 
\def\ov{\overline}
\def\ol{\ov}
\def\ul{\underline}
\def\lra{\longrightarrow}
\def\iso{\stackrel{\cong}\lra}

\def\C{\mathbb{C}}
\def\R{\mathbb{R}}
\def\N{\mathbb{N}}

\def\Z{\mathbb{Z}}

\def\V{\mathbb{V}}

\def\cA{\mathcal A}
\def\cB{\mathcal{B}}
\def\cC{\mathcal{C}}

\def\cE{\mathcal{E}}

\def\cG{\mathcal{G}}

\def\cI{\mathcal{I}}

\def\caL{\mathcal L}

\def\cN{\mathcal{N}}
\def\cO{\mathcal{O}}
\def\cP{\mathcal{P}}

\def\cR{\mathcal{R}}
\def\cS{\mathcal{S}}
\def\o{\mathfrak{r}}
\def\ru{\mathfrak{j}}
\def\d{\mathfrak{d}}

\def\fX{\mathfrak{X}}
\def\dom{\operatorname{dom}}

\def\reg{\operatorname{reg}}
\def\sing{\operatorname{sing}}
\def\sink{\operatorname{sink}}
\def\infem{\operatorname{inf}}
\def\sour{\operatorname{sour}}
\def\pssn{\operatorname{p-ssn}}
\def\tight{\operatorname{tight}}

\def\End{\operatorname{End}}

\def\tight{\mathrm{tight}}

\def\opq{\mathcal{O}^p(Q)}

\newcommand{\mspan}{\mbox{span}}

\begin{document}

\title{$L^p$ operator algebras associated with oriented graphs}

\author{Guillermo Corti\~nas}
\address{Dept.\ Matem\'atica-Inst.\ Santal\'o, FCEyN,
Universidad de Buenos Aires,
Ciudad Universitaria\\ (1428) Buenos Aires, Argentina}
\email{gcorti@dm.uba.ar}
\author{Ma. Eugenia Rodriguez}
\address{Departamento de Ciencias Exactas, Ciclo B\'asico Com\'un, 
Universidad de Buenos Aires\\
Ciudad Universitaria, (1428) Buenos Aires, Argentina}
\email{merodrig@dm.uba.ar}

\thanks{Both authors were supported by grants UBACyT 20021030100481BA and  PICT 2013-0454. Corti\~nas research was supported by Conicet
 and by grant MTM2015-65764-C3-1-P (Feder funds).}
\begin{abstract}
For each $1\le p<\infty$ and each countable oriented graph $Q$ we introduce an $L^p$-operator algebra $\cO^p(Q)$ which contains the Leavitt path $\C$-algebra
$L_Q$ as a dense subalgebra and is universal for those $L^p$-representations of $L_Q$ which are spatial in the sense of N.C. Phillips. For $\cR_n$ the graph with one vertex and $n$ loops ($2\le n\le \infty$), $\cO^p(\cR_n)=\cO^p_n$, the $L^p$-Cuntz algebra introduced by Phillips. If $p\notin\{1,2\}$ and $\cS(Q)$ is the inverse semigroup generated by $Q$,  $\cO^p(Q)=F_\tight^p(\cS(Q))$ is the tight semigroup  $L^p$-operator algebra introduced by Gardella and Lupini. We prove that $\cO^p(Q)$ is simple as an $L^p$-operator algebra if and only if $L_Q$ is simple, and that in this case it is isometrically isomorphic to the closure $\ol{\rho(L_Q)}$ of the image of any nonzero spatial $L^p$-representation $\rho:L_Q\to\caL(L^p(X))$. We also show that if $L_Q$ is purely infinite simple and $p\ne p'$, then there is no nonzero continuous homomorphism $\cO^p(Q)\to\cO^{p'}(Q)$. Our results generalize some similar results obtained by Phillips for $L^p$-Cuntz algebras.
\end{abstract}

\date{}
\maketitle

\section{Introduction}\label{sec:intro}

Let $Q$ be a countable oriented graph, let $Q^0$ and $Q^1$ be the sets of vertices and edges,  and let $L_Q$ be the Leavitt path $\C$-algebra. For $1\le p<\infty$ we call a representation $\rho:L_Q\to\caL(L^p(X))$ spatial if
$X$ is a $\sigma$-finite measure space and $\rho$ maps the elements of $Q^0\sqcup Q^1\sqcup (Q^1)^*$ to partial isometries which are spatial in the sense of \cite{chris1}*{Definition 6.4}. Each spatial representation $\rho$ induces a seminorm on $L_Q$ via $||a||_\rho=||\rho(a)||$; the supremum $\|\ \ \|$ of these seminorms is a norm (Proposition \ref{prop:sigmarep}) and we write $\cO^p(Q)$ for the completion of $(L_Q,\|\ \ \|)$. For $p\notin \{1,2\}$, $\cO^p(Q)$ agrees with the tight semigroup algebra introduced by Gardella and Lupini in \cite{gardelupi} (Proposition \ref{prop:opq=fps}). We prove the following.

\begin{thm}\label{thm:simpleintro}(Simplicity theorem)
Let $Q$ be a countable graph and let $1\le p<\infty$, $p\ne 2$. The following are equivalent.
\item[i)] $L_Q$ is simple.
\item[ii)] Every nonzero spatial $L^p$-representation of $L_Q$ is injective.
\item[iii)] Every nondegenerate contractive nonzero $L^p$-representation of $\cO^p(Q)$ is injective.
\goodbreak
If furthermore we have either that $Q^0$ is finite or that $p>1$, then the above conditions are also equivalent to:
\item[iv)] For every $L^p$-operator algebra $B$, every contractive, nonzero homomorphism $\cO^p(Q)\to B$ is injective. 
\end{thm}

Condition iv) says that $\cO^p(Q)$ is simple as an $L^p$-operator algebra. Since every $L^p$-operator algebra is isometrically embedded in $\caL(L^p(X))$ for some $\sigma$-finite measure space $X$, simplicity is equivalent to the condition that every contractive nonzero representation $\rho:\cO^p(Q)\to \caL(L^p(X))$, degenerate or not, be injective. We show (using a classical result of And\^o \cite{ando} and a recent result of Gardella and Thiel \cite{gardethiel2}) that if either $Q^0$ is finite or $p>1$, then the restriction of $\rho$ to $L_Q$ factors through a nondegenerate spatial representation; this allows us to prove that iii)$\iff$iv). 

To prove Theorem \ref{thm:simpleintro} we first show the following uniqueness theorem.

\begin{thm}\label{thm:indepintro}(Uniqueness theorem) 
Let $Q$ be a countable graph such that $L_Q$ is simple. Let $1\le p<\infty$, $X$ a $\sigma$-finite measure space and $\rho:L_Q\to \caL(L^p(X))$ a nonzero spatial representation. Then the canonical map $\cO^p(Q)\to \ol{\rho(L_Q)}$ is an isometric isomorphism. 
\end{thm}

Specializing Theorem \ref{thm:indepintro} to the case when $Q$ has only one vertex recovers N.C. Phillips' uniqueness result for $L^p$-analogues of Cuntz algebras  \cite{chris1}*{Theorem 8.7}. We also show (Theorem \ref{thm:opoq}) that if $Q$ and $Q'$ are countable graphs with $L_Q$ purely infinite simple and $1\le p\ne p'<\infty$  then it is often the case that no nonzero continuous homomorphism $\cO^p(Q)\to\cO^{p'}(Q')$ exists. For example this is the case when $L_{Q'}$ is simple. In particular, we have the following.

\begin{thm}\label{thm:opoqintro} 
Let $Q$ be a countable graph. If $L_Q$ is purely infinite simple then there is no nonzero continuous homomorphism $\cO^p(Q)\to\cO^{p'}(Q)$. 
\end{thm}

A similar result for $L^p$-Cuntz algebras was obtained by N.C. Phillips in \cite{chris1}*{Theorem 9.2}. 

\smallskip

The rest of this paper is organized as follows. In Section \ref{sec:leav} we recall some definitions and basic facts on Leavitt path algebras and prove  some elementary technical lemmas. In Section \ref{sec:semigrp} we show (Lemma \ref{lem:tight}) that $L_Q$ is the universal algebra for tight algebraic representations of the inverse semigroup $\cS(Q)$ generated by $Q$. Spatial representations of the Leavitt path algebra 
$L_Q$ of a countable graph $Q$ are introduced in Section \ref{sec:spatrep}. We give examples of such representations and show in Proposition \ref{prop:sigmarep} that for every countable $Q$ and $1\le p<\infty$ there is an injective, nondegenerate spatial representation $L_Q\to\caL(\ell^p(\N))$. Spatial representations of matrix algebras 
$M_nL_Q$ for $1\le n\le\infty$ are considered in Section \ref{sec:matspat} and it is shown that they are the same as spatial representations of the Leavitt path algebra over the graph $M_nQ$ (Remark \ref{rem:matspatleav}) and that any such representation is equivalent to the matricial amplification $M_n\rho$ of a spatial representation $\rho$ of $L_Q$ (Lemma \ref{lem:matspat}). Section \ref{sec:crite} is concerned with a characterization of spatiality of representations in terms of norm estimates. We prove a spatiality criterion which we shall presently explain. The subalgebra $(L_Q)_{0,1}=\mspan\{v\in Q^0,ee^*, e\in Q^1\}\subset L_Q$ is a direct sum of --possibly infinite dimensional-- matrix algebras and is thus naturally equipped with a canonically equipped with an $L^p$-operator norm. The spatiality criterion, Theorem \ref{thm:contraspat} --which generalizes \cite{chris1}*{Theorem 7.7}-- says that if $p\in [1,\infty)$, $p\ne 2$, then a nondegenerate representation $\rho:L_Q\to\caL(L^p(X))$ is spatial if and only if its restriction to $(L_Q)_{0,1}$ is contractive and $\|\rho(x)\|\le 1$ for every $x\in  Q^1\coprod (Q^1)^*$ (cf. \cite{chris1}*{Theorem 7.7}). Along the way we also prove a spatiality criterion for nondegenerate $L^p$-representations of matricial algebras (Proposition \ref{prop:mat}) which generalizes \cite{chris1}*{Theorem 7.2}. Both spatiality criteria fail to be true if the nondegeneracy hypothesis is dropped (see Remark \ref{rem:ojo}). In Section \ref{sec:opq} we define $L^p$-operator algebras and introduce the $L^p$-operator algebra $\cO^p(Q)$. By definition, any spatial representation of $L_Q\to \caL(L^p(X))$ factors uniquely through a contractive representation $\cO^p(Q)\to\caL(L^p(X))$ ($1\le p<\infty$). Moreover we prove, using the spatiality criterion of Section \ref{sec:crite}, that for $p\ne 2$, any nondegenerate contractive representation $\cO^p(Q)\to \caL(L^p(X))$ induces a nondegenerate spatial representation $L_Q\to\caL(L^p(X))$ (Theorem \ref{thm:contraspat15}). Using a result of E. Gardella and H. Thiel from \cite{gardethiel2}, we show that if moreover $p\ne 1$, then the nondegeneracy hypothesis may be dropped (Theorem \ref{thm:contraspat2}). We also show, using the material of Section \ref{sec:semigrp}, that if $p\notin\{1,2\}$ then $\cO^p(Q)$ is the same as the $L^p$-algebra $F_\tight^p(\cS(Q))$ introduced by E. Gardella and M. Lupini in \cite{gardelupi}, which is universal for tight $L^p$-spatial representations of $\cS(Q)$. In the next section we show that adding heads and tails to a graph $Q$ to obtain a new graph $Q'$ without sources, sinks or infinite emitters results in an isometric inclusion $\cO^p(Q)\to\cO^p(Q')$ (Corollary \ref{coro:seminorms}). Section \ref{sec_O(Q)} is devoted to the proof of Theorem \ref{thm:indepintro} (Theorem \ref{thm:indep}). The technical result of the previous section is used here to reduce the proof to the case of graphs without sources, sinks or infinite emmitters. After this reduction, the strategy of proof is similar to that of \cite{chris1}*{Theorem 8.7}, although it requires several nontrivial technical adjustments. Simplicity Theorem \ref{thm:simpleintro} is proved in Section \ref{sec:simplicity}. In fact we prove in Theorem \ref{thm:simple} that the simplicity of $L_Q$ is equivalent not only to the conditions of Theorem \ref{thm:simpleintro}, but also to other more restrictive conditions, e.g. that every nondegenerate spatial nonzero representation $L_Q\to\caL(\ell^p(\N))$ be injective. The last section of this article is Section \ref{sec:opoq}, where we prove Theorem \ref{thm:opoq}, of which Theorem \ref{thm:opoqintro} is a particular case.

\begin{nota}
In this paper $\N=\Z_{\ge 1}$ and $\N_0=\Z_{\ge 0}$. All algebras, vector spaces, and tensor products are over $\C$. All identities pertaining measure spaces are to be interpreted up to sets of measure zero. For example we say that a family $\{X_n\}_{n\ge 1}$ of measurable sets in a measurable space $X=(X,\cB,\mu)$ is \emph{disjoint} if $X_n\cap X_{m}$ has measure zero for all $n\ne m$, and write $\coprod_nX_n$ for their union. In case the latter agrees with $X$ up to measure zero, we write $X=\coprod_nX_n$. This reflects the fact that under the above hypothesis $(X,\cB,\mu)$ is equivalent to set theoretic coproduct $\coprod_nX_n$ equipped with the $\sigma$-algebra generated by $\coprod_n\cB_n$ and the measure induced by the sequence of measures $\{\mu_{|X_n}\}$. 
We write $L^0(X)$ for the vector space of classes of measurable functions $X\to \C$. 
\end{nota}

\smallskip

\noindent{\em Acknowledgements.} This article has evolved from the PhD thesis of the second named author \cite{eugethesis}. We are indebted to Chris Phillips for discussions on his paper \cite{chris1}. Thanks also to our colleague Daniel Carando for several useful discussions and references on $L^p$-spaces. The first named author also wishes to thank Eusebio Gardella for an enlightening email exchange including several useful comments on a previous version of this paper.

\section{Graphs and Leavitt path algebras}\label{sec:leav}

An \emph{oriented graph} or \emph{quiver}  $Q=(Q^0,Q^1,r,s)$ consists of sets $Q^0$ and $Q^1$ of vertices and edges, and range and source functions $r,s:Q^1\rightarrow Q^0$ . We say that $Q$ is \emph{finite} or \emph{countable} if $Q^0$ and $Q^1$ are both finite or countable. A vertex $v\in Q^0$ is an \emph{infinite emitter} if $s^{-1}(v)$ is infinite, and is a \emph{sink} if $s^{-1}(v)=\emptyset$. A vertex is \emph{singular} if it is either a sink or an infinite emitter. We write $\sing(Q)=\sink(Q)\cup\infem(Q)\subset Q^0$ for the set of singular vertices and $\reg(Q)=Q^0\setminus\sing(Q)$. We call $Q$ \emph{singular} 
if $\sing(Q)\ne\emptyset$ and  \emph{nonsingular} (or \emph{regular}) otherwise.  We call $Q$ \emph{row-finite} if it has no infinite emitters. A vertex $v$ is a \emph{source} if $r^{-1}(v)=\emptyset$; we write $\sour(Q)\subset Q^0$ for the set of sources.

Since all our graphs will be oriented, we shall use the term graph to mean oriented graph.

A \emph{path} $\alpha$ is a (finite or infinite) sequence of edges $\alpha=e_1\ldots e_i\ldots$ such that $r(e_i)=s(e_{i+1})$  $(i\ge 1)$. For such $\alpha$, we write $s(\alpha)=s(e_1)$; if $\alpha$ is finite of \emph{length} $l$, we put $|\alpha|=l$ and $r(\alpha)=r(e_l)$. Vertices are considered as paths of length $0$. A finite path $\alpha$ is \emph{closed} if $s(\alpha)=r(\alpha)$. A closed path $\alpha=\alpha_1\dots\alpha_n$ is a \emph{cycle} if in addition $s(e_i)\ne s(e_j)$ if $i\ne j$. 
Let $\mathcal{P}=\cP(Q)$ be the set of finite paths, and let $\mathcal{P}_n$ be the set of paths of length $n$. Thus, 
\begin{equation}\label{pcoprod}
\mathcal{P}=\coprod_{n\in \mathbb{N}_0}\mathcal{P}_n.
\end{equation}

We consider the following preorder in $\cP$:

\begin{equation}\label{orderpath}
    		\alpha\le\beta \iff \exists\ \gamma \text{ such that } r(\beta)=s(\gamma) \text{ and } \alpha=\beta\gamma.
\end{equation}
Observe that \eqref{orderpath} also makes sense when $\alpha$ is an infinite path.

\begin{defi}\label{defi:def_leav}
Let $Q$ be a graph. The \emph{Leavitt path algebra} $L_Q$ is the quotient of the free $\C$-algebra on $Q^0\cup Q^1\cup(Q^1)^*$, modulo the following relations:

\begin{itemize}
 \item $vv'=\delta_{v,v'}v\ \forall\ v,v'\in Q^0,$

 \item $s(e)e=er(e)=e\ \forall\ e\in Q^1,$

 \item $r(e)e^*=e^*s(e)=e^*\ \forall\ e\in Q^1,$

 \item (CK1) $e^*e'=\delta_{e,e'}r(e)\ \forall\ e,e'\in Q^1,$

 \item (CK2) $v=\displaystyle{\sum_{\{e\in Q^1:s(e)=v\}}}ee^*$, if $v\in\reg(Q)$.
\end{itemize}
\end{defi}
The Leavitt path algebra is a $*$-algebra with involution determined by $v\mapsto v$, $e\mapsto e^*$. It has a $\Z$-grading where vertices have degree zero, edges have degree $1$, and  
$|e^*|=-1$ for $e\in Q^1$ (\cite{libro_Pere}*{Corollary 2.1.5}). We write 
\begin{equation}\label{grading}
(L_Q)_n=\mspan\{\alpha\beta^*: |\alpha|-|\beta|=n\}
\end{equation} 
for the $n$-th homogeneous component with respect to this grading. 

The elementary lemmas below shall be used later in the article. 

\begin{lemma}\label{lem:escr_sum}
Let $Q$ be a nonsingular graph and $a_1,\ldots,a_m\in L_Q$. Then there exist $n\in\mathbb{N}$, a finite set $F\subset \mathcal{P}$, and finitely supported functions $\lambda^i:F\times\mathcal{P}_n\rightarrow\mathbb{C}$, $(\alpha,\beta)\mapsto\lambda_{\alpha,\beta}^i$ ($i=1,\ldots m$, $\alpha\in F$, $\beta\in \mathcal{P}_n$), such that $$a_i=\sum_{\alpha\in F}\sum_{\beta\in \mathcal{P}_n}\lambda^i_{\alpha,\beta}\alpha\beta^*,\ \forall i=1,\ldots,m.$$
\end{lemma}

\begin{proof}
For each $i=1,\ldots,m$, we may write $a_i=\sum_{j=1}^{n_i}\lambda^i_j\alpha_j^i{\beta_j^i}^*$ with paths $\beta_j^i$ of length $n:=\displaystyle{\max_{i,j}}\{|\beta_j^i|\}$, using relation (CK2) of Definition \ref{defi:def_leav}. Put $F_i:=\{\alpha_j^i:j=1,\ldots,n_i\}$, $G_i:=\{\beta_j^i:j=1,\ldots,n_i\}$ and $F:=\displaystyle{\bigcup_{i=1}^mF_i}$. Rewriting the sums for each $i$, we have $a_i=\sum_{\alpha\in F}\sum_{\beta\in\mathcal{P}_n}\lambda_{\alpha,\beta}^i\alpha\beta^*$ with $\lambda_{\alpha,\beta}^i=0$ if $\alpha\notin F_i$ or $\beta\notin G_i$.
\end{proof}

\begin{lemma}\label{lem_exist_repres}
 Let $Q$ be a graph, $B$ a $\C$-algebra, and $\rho: L_Q\rightarrow B$ a homomorphism. Let $u:=\{u_v\}_{v\in Q^0}\subset B$ such that $u_v$ is invertible in $\rho(v)B\rho(v)$  $(v\in Q^0)$. Then there is a unique homomorphism $\rho_u: L_Q\rightarrow B$ such that

 $$\rho_u(e)=u_{s(e)}\rho(e),\ \ \rho_u(e^*)=\rho(e^*)u_{s(e)}^{-1}\ \ \ \textrm{and} \ \ \rho_u(v)=\rho(v)\ \ (\forall e\in Q^1,\ v\in Q^0).$$
\end{lemma}

\begin{proof} One checks that the elements $\rho(x)$, $x\in Q^0\cup Q^1\cup (Q^1)^*$ satisfy the relations of Definition \ref{defi:def_leav}.
\end{proof}

\section{Leavitt path algebras and semigroups}\label{sec:semigrp}

Let $Q$ be a graph and $\cP=\cP(Q)$ the set of finite paths. Write 
\begin{equation}\label{semiQ}
\cS=\cS(Q)=\{0\}\cup\{\alpha\beta^*:\alpha,\beta\in \cP\}\subset L_Q.
\end{equation}
$\cS$ is the inverse semigroup associated with $Q$. The \emph{Cohn} algebra of $Q$ is the semigroup algebra $C_Q=\C[\cS]$ of $\cS$; its elements are the finite linear combinations of the elements of $\cS$ with multiplication induced by that of $\cS$. Observe that $L_Q$ is the quotient of $C_Q$ modulo the relation $CK2$. Consider
\[
\cS\supset\cE=\{0\}\cup\{\alpha\alpha^*:\alpha\in \cP\}
\]
the sub-semigroup of idempotent elements. The set $\cE$ is partially ordered by $p\le q\iff pq=p$ and is a semilattice for this partial order. Observe that for the order of paths defined in \eqref{orderpath}, the bijection $\cP\to \cE\setminus\{0\}$, $\alpha\mapsto \alpha\alpha^*$ is a poset isomorphism. Note also that $p,q\in\cE$ are incomparable if and only if $pq=0$. Let $p\in \cE$ and $Z\subset \{q\in\cE:q\le p\}$. 
We call $Z$ a \emph{cover} of $p$ if for every $q\le p$ there exists $z\in Z$ such that $zq\ne 0$. A \emph{representation} of $\cS$ in a vector space $\V$ is a semigroup homomorphism $\rho:\cS\to(\End(\V),\circ)$, where the latter is the set of linear endomorphisms considered as a semigroup under composition. The image of $\cE$ under a representation $\rho$ generates a boolean algebra $\cB_\rho$ with operations $p\wedge q=pq$, $p\vee q=p+q-pq$. By \cite{exel}*{Proposition 11.8}, the boolean representation $\rho:\cE\to\cB_\rho$ is \emph{tight} in the sense of \cite{exel}*{Definition 11.6} if and only if for every $p\in\cE$ and every finite cover $Z$ of $p$, we have
\begin{equation}\label{condi:tight}
\bigvee_{z\in Z}\rho(z)=\rho(p).
\end{equation}
Following \cite{exel}*{Definition 13.1}, we call the representation $\rho$ of $\cS$ tight if its restriction to $\cE$ is tight.

Although the following lemma is well-known to experts, we have not been able to find it explicitly stated in the literature, so we include it here with proof. The particular case of Lemma \ref{lem:tight} when $Q$ has a single vertex is \cite{gardethiel2}*{Lemma 7.5}. See also \cite{stein2}*{Corollary 5.3}.

\begin{lem}\label{lem:tight}
Let $\rho:\cS(Q)\to \End(\V)$ be a representation. Then $\rho$ is tight if and only if it extends to an algebra homomorphism $L_Q\to \End(\V)$. 
\end{lem}
 \begin{proof} If $v\in\reg(Q)$, then $Z=\{ee^*:e\in Q^1, s(e)=v\}$ is a finite cover of $v$ and the supremum in \eqref{condi:tight} equals $\sum_{e\in Z}\rho(ee^*)$. It follows that if $\rho$ is tight then it extends to an algebra homomorphism $L_Q\to\End(\V)$. Assume conversely that $\rho$ extends to $L_Q$. We have to prove that \eqref{condi:tight} holds. Since the supremum in \eqref{condi:tight} depends only on the maximal elements of $Z$ and any two of these are incomparable we may assume that no two distinct elements of $Z$ are comparable. Hence
\[
\bigvee_{z\in Z}\rho(z)=\sum_{z\in Z}\rho(z).
\] 
If $\alpha\in\cP$ and $r(\alpha)=v$, then $W=\alpha^*Z\alpha$ is a cover of $v$ and $\sum_{z\in\Z}=\alpha\sum_{w\in W}w\alpha^*$. Hence we may further assume that $\alpha=v$. We must then prove that the following identity holds in $L_Q$
\[
\sum_{z\in Z}z=v
\]
for each finite cover $Z$ of $v$ in which no two distinct elements are comparable. We do this by induction on $n=m(Z)=\max\{|\alpha|:\alpha\alpha^*\in Z\}$. 
For $n=0$ this is trivial. Assume $n\ge 1$ and let $A=\{\alpha\in\cP_n:\alpha\alpha^*\in Z\}$. Each $\alpha\in A$ writes uniquely as $\tilde{\alpha}e_\alpha$ where $|\tilde{\alpha}|=n-1$ and $e_\alpha\in Q^1$. For $w\in B:=\{s(e_\alpha):\alpha\in A\}$, put $C_w=\{e_\alpha:s(e_\alpha)=w\}$; because $Z$ is a cover, $C_w=s^{-1}(w)$. Hence 
\begin{align*}
\sum_{\alpha\in A}\alpha\alpha^*=&\sum_{\beta\in \tilde{A}}\sum_{\tilde{\alpha}=\beta}\alpha\alpha^*\\
                                =&\sum_{\beta\in \tilde{A}}\sum_{s(e)=r(\beta)}\beta ee^*\beta^*\\
																=&\sum_{\beta\in\tilde{A}}\beta\beta^*.
\end{align*}
Let $Z'=(Z\setminus A)\cup \tilde{A}$; then $m(Z')=n-1$, any two distinct elements of $Z'$ are incomparable, and by the calculation above, $\sum_{z'\in Z'}z'=\sum_{z\in Z}z$. This concludes the proof. 
\end{proof}

\section{Spatial representations of \(L_Q\)}\label{sec:spatrep}

Let $E$ be a Banach space. We write $\caL(E)$ for the Banach algebra of bounded linear maps $E\rightarrow E$.
A \emph{representation} of $L_Q$ on $E$ is an algebra homomorphism $\rho:L_Q\to \caL(E)$. We say that $\rho$ is \emph{nondegenerate} if $\rho(L_Q)E\subset E$ is dense. In this paper we shall be mostly concerned
with \emph{$L^p$-representations}, that is, with representations on Banach spaces of the form $L^p(X)$ $(1\le p<\infty)$ where $X=(X,\cB,\mu)$ is a $\sigma$-finite measure space. If $A\in\cB$, we write $P(A)$ for the set of subsets of $A$ and consider 
$A$ as a measure space with $\sigma$-algebra $\cB_A:=\cB\cap P(A)$ and measure $\mu_{|\cB_A}$; thus
\[
A=(A,\cB_A,\mu_{|\cB_A}).
\]
We write $\cN(\mu)=\{A\in \cB:\mu(A)=0\}$, $\cB_\mu=\cB/\cN(\mu)$. 

In what follows, we need to borrow several definitions from \cite{chris1}, pertaining to (partial) isometries between $L^p$-spaces. 

Let $X=(X,\cB,\mu)$ and $(Y,\cC,\nu)$ be $\sigma$-finite measure spaces. A \emph{measurable set transformation} from $X$ to $Y$ is homomorphism of $\sigma$-algebras $S:\cB_\mu\to\cC_\nu$. If $S$ is bijective, then $S_*(\mu)=\mu S^{-1}$ is a $\sigma$ finite measure on $\cC$, absolutely continuous with respect to $\nu$. By \cite{chris1}*{Proposition 5.6}, there is also a map $S_*:L^0(X)\to L^0(Y)$ such that $S_*(\chi_E)=\chi_{S(E)}$ ($E\in\cB_\mu$). Let $1\le p<\infty$; to a bijective measurable set transformation $S$ from $X$ to $Y$ and a measurable function $h:Y\to \C$ such that $|h(x)|=1$ for almost every $x\in B$ one associates an isometric isomorphism $u:L^p(X)\to L^p(Y)$ as follows:
\begin{equation}\label{spatsyst}
u(\xi)(y)=h(y)([\frac{dS_*(\mu)}{d\nu(y)}])^{1/p}S_*(\xi)(y) \quad (\xi\in L^p(X)).
\end{equation}
An isometric isomorphism $u:L^p(X)\to L^p(Y)$ is called \emph{spatial} if there exist $S$ and $h$ such that $u$ is of the form \eqref{spatsyst}. If $p\ne 2$, then every isometric isomorphism in $\caL(L^p(X),L^p(Y))$ is spatial, by the Banach-Lamperti theorem (\cite{chris1}*{Theorem 6.9 and Lemma 6.15}). 
A partial isometry $s:L^p(X)\to L^p(X)$ is 
\emph{spatial} if there are $A,B\in \cB_\mu$, called respectively the \emph{domain} and the \emph{range support} of $s$, such that for the projection $\pi_A:L^p(X)\to L^p(A)$ and the inclusion $\iota_B:L^p(B)\to L^p(X)$  we have a factorization 
\begin{equation}\label{spatiso}
s=\iota_B u\pi_A
\end{equation}
 where $u:L^p(A)\to L^p(B)$ is a spatial isometric isomorphism. If $S$ and $h$ are as in \eqref{spatsyst} we call $s$ the spatial partial isometry associated with the \emph{spatial system} $(S,A,B,h)$; $S$ and $h$ are the \emph{spatial realization} and the \emph{phase factor} of the spatial system. The \emph{reverse} of the spatial partial isometry \eqref{spatiso} is the spatial partial isometry $t=\iota_A u^{-1}\pi_B$. If $p=2$, then the reverse of a spatial partial isometry $s$ is just its adjoint $t=s^*$. 

\begin{exa}\label{exa:idempotent} 
Let $X=(X,\cB,\mu)$ be a $\sigma$-finite measure space. Let $E\in\cB$  and let $\chi_E$ be the characteristic function. Then the canonical projection $\pi_E:L^p(X)\to L^p(E)\subset L^p(X), \pi_E(\xi)=\chi_E\xi$ is a spatial partial isometry with spatial system $(Id_{\cB_E},E,E,1)$. Every idempotent spatial partial isometry is of this form, by \cite{chris1}*{Lemma 6.18}. 
\end{exa}

\begin{rem}\label{rem:idempotent}
Spatial partial isometries in general and spatial idempotents in particular have norm $1$. However the converse does not hold. For example, 
\[
\left(\begin{matrix}1/2&1/2\\ 1/2& 1/2\end{matrix}\right)\in M_2=\cB(\ell^p(\{1,2\})
\]
is a norm one idempotent that is not spatial in our sense (which is that of \cite{chris1}) for any $p\ge 1$ (\cite{chris1}*{Example 7.3}). However it is self-adjoint and therefore $2$-spatial in the sense of \cite{gardelupi}*{Definition 4.6}.
\end{rem}

A representation $\rho:L_Q\to \caL(L^p(X))$ is \emph{spatial} if for each $v\in Q^0$, $\rho(v)$ is a spatial idempotent and for each $e\in Q^1$, $\rho(e)$ is a spatial partial isometry with reverse $\rho(e^*)$. 
If $\rho$ is spatial then $\rho(x)$ is spatial for every $x\in\cS(Q)$, whence by Lemma \ref{lem:tight} a spatial representation of $L_Q$ is the same as a tight \emph{spatial representation of $\cS(Q)$}, that is, a tight representation of $\cS(Q)$ which takes values in the inverse semigroup $\cS(L^p(X))$ of spatial partial isometries. 

\begin{rem}\label{rem:spat*}
As we explained above, the reverse of a spatial isometry $s\in L^2(X)$ is just its adjoint. Hence any spatial representation $L_Q\to\caL(L^2(X))$ is a $*$-representation. The converse does not hold. For example $\C$ is the Leavitt path algebra of the graph consisting of a single vertex and no edges, and the representation $\rho:\C\to M_2=\caL(\ell^2(\N))$ that sends $1$ to the self-adjoint idempotent of Remark \ref{rem:idempotent} is a $*$-representation that is not spatial in our sense.
\end{rem}

\begin{rem}\label{rem:comparables}
If $\rho$ is spatial and $\alpha,\beta\in\cP(Q)$ are paths with $r(\alpha)=r(\beta)$, then $\rho(\alpha\beta^*)$ is a spatial partial isometry. In particular, $\rho(\alpha\alpha^*)$ is an idempotent spatial partial isometry, and thus by Example \ref{exa:idempotent}, there is  
$X_\alpha\in \cB$ such that $\rho(\alpha\alpha^*)$ is the canonical projection $\pi_{X_\alpha}:L^p(X)\to L^p(X_\alpha)\subset L^p(X)$. If $S_\alpha$ is the measurable set transformation of $\rho(\alpha)$ then $X_\alpha=S_\alpha(X_{r(\alpha)})$, so the spatial system of $\rho(\alpha)$ is of the form
$$(S_\alpha, X_{r(\alpha)},X_\alpha,g_\alpha)$$ for some $g_\alpha:X_\alpha\to \C$ such that $|g(x)|=1$ for almost all $x\in X_\alpha$. 
If $\alpha\geq\beta$, say $\beta=\alpha\gamma$, then 
$X_{\beta}\subset X_{\alpha}$ because $X_{\beta}=S_{\alpha}(X_{\gamma})\subset S_{\alpha}(X_{r(\alpha)})= X_{\alpha}$.
On the other hand if $\alpha$ and $\beta$ are not comparable then $X_{\alpha}$ and $X_{\beta}$ are disjoint.
In particular, for each $v\in Q^0$ the family $\{X_e:s(e)=v\}\subset\cB\cap P(X_v)$ is disjoint, and if $v$ is regular its union is the whole $X_v$:
\begin{equation}\label{decompo:regv}
X_v=\coprod_{e\in s^{-1}(v)}X_e  \qquad (v\in\reg(Q)). 
\end{equation}
 It follows from \eqref{decompo:regv} that if $Q$ is nonsingular then for each $l\ge 0$ we have
\begin{equation}\label{decompo:regQ}
X_v=\coprod_{\alpha\in v\cP_l(Q)}X_\alpha\qquad (Q^0=\reg(Q)).
\end{equation}
Conversely, if we are given disjoint families $\{X_v:v\in Q^0\}\subset\cB$ and $\{X_e:e\in Q^1, s(e)=v\}\subset \cB\cap P(X_v)$ $(v\in Q^0)$ satisfying \eqref{decompo:regv} and a family $\{s_e:e\in Q^1\}$ of spatial partial isometries in $\caL(L^p(X))$ with range and source projections $\pi_{X_e}$ and $\pi_{X_{v}}$, then we have a unique algebra homomorphism $\rho:L_Q\to\caL(L^p(X))$ mapping $\rho(v)=\pi_{X_v}$, $\rho(e)=s_e$, and sending $e^*$ to the reverse of $s_e$. 
\end{rem}

\begin{lem}\label{lem:nondeg} Let $X$ be a $\sigma$-finite measure space. A spatial representation $\rho:L_Q\to \caL(L^p(X))$ is nondegenerate if and only if  
\begin{equation}\label{decompo:nondeg}
X=\coprod_{v\in Q^0}X_v.
\end{equation}
\end{lem}
\begin{proof} Immediate from the fact that 
\[
\rho(L_Q)L^p(X)=\sum_{v\in Q^0}\rho(v)L^p(X)=\bigoplus_{v\in Q^0}L^p(X_v).
\]
\end{proof}

It follows from \eqref{decompo:regQ} and Lemma \ref{lem:nondeg} that if $Q$ is nonsingular and $\rho$ is nondegenerate, then for each $l\ge 0$ we have
\begin{equation}\label{decompo:nondegreg}
X=\coprod_{\alpha\in\cP_l(Q)}X_\alpha.
\end{equation}

\begin{lem}\label{lem:nondegenerate}
Let $Q$ be a graph, $1\le p<\infty$, $X=(X,\cB,\mu)$ a $\sigma$-finite measure space, and $\rho:L_Q\to \caL(L^p(X))$ a spatial representation. Then there are $X'\in\cB$ and a nondegenerate spatial representation $\rho':L_Q\to \caL(L^p(X'))$ such that $\rho$ factors as $\rho'$ followed
by the inclusion $\caL(L^p(X'))\subset\caL(L^p(X))$. 
\end{lem}
\begin{proof} Put $X'=\coprod_{v\in Q^0}X_v$.
\end{proof}

\begin{exa}\label{ex:tight} Let $Q$ be a graph, and let 
\begin{equation}\label{fX}
\fX=\fX_Q=\{\alpha: \text{ infinite path in } Q\}\cup\{\alpha\in\cP:r(\alpha)\in\sing(Q)\}.
\end{equation}
For $\alpha\in\cP$, let 
\[
\fX\supset Z_{\alpha}=\{x\in \fX: \alpha\ge x\}=\alpha\fX.
\]
The sets $Z_{\alpha}$ are the basis of a topology which makes it a locally compact Hausdorff space; modulo our different conventions for ranges and sources, this is the space considered in \cite{orlosim}*{page 3}. The inverse semigroup $\cS=\cS(Q)$ acts on $\fX$ by partial homeomorphisms; an element $u=\alpha\beta^*\in \cS$ acts on $\fX$ with domain $Z_\beta$ and range $Z_\alpha$ via 
\begin{equation}\label{actsemi}
\alpha\beta^*(\beta x)=\alpha x.
\end{equation}  
Let $\cB$ the the $\sigma$-algebra of all Borel subsets of $\fX$.  The semigroup $\cS$ of \eqref{semiQ} acts on $\fX$ via \eqref{actsemi}. If $\alpha,\beta\in\cP$ with $r(\alpha)=r(\beta)$, then
\begin{equation}\label{act:fX}
S_{\alpha\beta^*}:\cB_{|Z_\beta}\to \cB_{|Z_\alpha},\quad A\mapsto \alpha\beta^*(A)
\end{equation}
is a bijective homomorphism of $\sigma$-algebras. Let $\mu$ be a measure on $\cB$; $\mu$ is \emph{quasi-invariant} under $\alpha\beta^*$ if $\mu_{|Z_\beta}$ and $\mu_{|Z_\alpha}\circ\beta\alpha^*$ are equivalent measures (that is, if they are absolutely continuous with respect to each other); $\mu$ is quasi-invariant under $\cS$ it is quasi-invariant under any element of $\cS$. One can show that $\fX$ always has a $\sigma$-finite measure that is quasi-invariant under $\cS$. For example, in case $\fX$ is countable we can take $\mu$ to be the counting measure. Assume that $\mu$ is a $\sigma$-finite measure on the Borel subsets of $\fX$, quasi-invariant under $\cS$, and let $s_{\alpha\beta^*}$ be the spatial isometry of 
\eqref{spatsyst} with spatial realization $S=S_{\alpha\beta^*}$ and constant phase factor $h=1$. Then 
\[
\cS\to \caL(L^p(\fX,\mu)),\quad \alpha\beta^*\mapsto s_{\alpha\beta^*}
\]
is a tight nondegenerate spatial representation of $\cS$ and thus induces a nondegenerate spatial representation $\rho_\mu:L_Q\to\caL(L^p(\fX,\mu))$. In general, $\rho_\mu$ is not injective. For example, if $Q$ consists of one vertex and one loop, then $L_Q\cong \C[t,t^{-1}]$ and $\rho_\mu$ is $1$-dimensional. 
\end{exa}

\begin{constru}\label{constru}
Let $X$ be a countable set, and let $\cI(X)$ be the inverse semigroup of all partially defined injections 
\[
X\supset\dom f\overset{f}\lra X.
\] 
Let $Q$ be a countable graph, $\cS=\cS(Q)$ its associated inverse semigroup and $S:\cS\to\cI(X)$ a semigroup homomorphism. For each $\alpha\in\cP=\cP(Q)$, set 
$X_\alpha=\dom(S_\alpha)$. We shall assume that $S$ is tight, i.e. that the identities \eqref{decompo:regv} and \eqref{decompo:nondeg} are satisfied. Let $\cG=\cG(\cS,X)$ be the groupoid of germs, as defined in \cite{exel}*{Section 4}. The elements of $\cG$ are equivalence classes
$[\alpha\beta^*,x]$ where $r(\alpha)=r(\beta)$, $x\in X_\beta$; the equivalence relation is determined by the prescription that $[\alpha\beta^*,x]=[\alpha\gamma \gamma^*\beta^*,x]$ for any $\gamma\in \cP$ with $s(\gamma)=r(\alpha)$. For $\alpha\beta^*\in\cS\setminus\{0\}$, put 
\[
\Theta_{\alpha,\beta}=\{[\alpha\beta^*,x]: x\in X_\beta\}\subset\cG.
\]
Let $\cA(\cG)\subset\operatorname{map}(\cG,\C)$ be the linear subspace generated by the characteristic functions $\chi_{\Theta_{\alpha,\beta}}$, $(\alpha\beta^*\in\cS\setminus\{0\})$. One checks that $\cA(\cG)$ is an algebra under the convolution product (it is in fact the Steinberg algebra of $\cG$ \cite{stein}) and that 
\begin{equation}\label{map:elfi}
\psi:L_Q\to \cA(\cG),\quad \phi(\alpha\beta^*)=\chi_{\Theta_{\alpha,\beta}}
\end{equation}
is an algebra homomorphism. Let 
\begin{equation}\label{map:laL}
L:\cA(\cG)\to \caL(\ell^p(\cG)), \quad L(f)(\xi)(h)=\sum_{ g\in\cG}f(g)\xi(g^{-1}h)
\end{equation}
This is well-defined because the domain and range functions are injective on each $\Theta_{\alpha,\beta}$. One checks that $L$ is a monomorphism. Consider the composite 
\begin{equation}\label{map:laro}
\rho=L\psi:L_Q\to \caL(\ell^p(\cG)).
\end{equation}
Let $\alpha\beta^*\in\cS(Q)$ and consider the following subsets of $\cG$:
\[
A=\{[\gamma\delta^*,\delta x]: \beta\ge \gamma x\},\qquad
B=\{[\alpha\beta^*\gamma\delta^*,\delta x]:\beta\ge\gamma x\}.
\]
The map 
\begin{gather*}
A\to B,\\
[\gamma\delta^*,\delta x]\mapsto [\alpha\beta^*,\gamma x][\gamma\delta^*,\delta x]=[\alpha\beta^*\gamma\delta^*,\delta x]
\end{gather*}
is bijective and thus induces a cardinality preserving bijection $S_{\alpha,\beta}:P(A)\to P(B)$. One checks that $\rho(\alpha\beta^*)$ is the spatial isometry with spatial system $(S_{\alpha,\beta},A,B,1)$. Hence $\rho$ is a spatial, nondegenerate representation.
\end{constru}

\begin{lem}\label{lem:grpd} Assume that in Construction \ref{constru}, one has $X_v\ne\emptyset$ for all $v\in Q^0$. Then \eqref{map:elfi} is an isomorphism 
and \eqref{map:laro} is an injective, nondegenerate spatial representation.
\end{lem}
\begin{proof} Put $\cA(\cG)_n=\mspan\{\psi(\alpha\beta^*):|\alpha\beta^*|=n\}$; we have 
\begin{equation}\label{decompo:ag}
\cA(\cG)=\sum_n\cA(\cG)_n.
\end{equation}
Let $c:\cG\to \Z$, $c([\alpha\beta^*,x])=|\alpha\beta^*|$; note that the elements of
$\cA(\cG)_n$ are supported in $c^{-1}(\{n\})$. It follows from this that the sum in \eqref{decompo:ag} is direct. Moreover, because $c$ is a groupoid homomorphism, we have  $\cA(\cG)_n\cA(\cG)_m\subset\cA(\cG)_{n+m}$. Thus $\psi$ is a homogeneous homomorphism of graded algebras. But for $v\in Q^0$, $\psi(v)$ is the characteristic function of $\{[v,x]:x\in X_v\}$ which is nonempty by hypothesis, so $\psi(v)\ne 0$. By \cite{libro_Pere}*{Theorem 2.2.15} this implies that $\psi$ is an isomorphism. 
\end{proof}

\begin{prop}\label{prop:sigmarep} Let $Q$ be a countable graph. Then $L_Q$ has an injective, nondegenerate spatial representation 
$L_Q\to\caL(\ell^p(\N))$.
\end{prop}
\begin{proof} Let $X$ be any countably infinite set. Because $X$ is infinite and $\#Q^0\le\# X$, there exists a bijection $\phi:X\to Q^0\times X$. For $v\in Q^0$, set $X_v=\phi^{-1}(\{v\}\times X)$; observe that \eqref{decompo:nondeg} is satisfied by construction. Put $Q^1(v,-)=s^{-1}(\{v\})\subset Q^1$ and let
\[
R_v=\left\{\begin{matrix}Q^1(v,-)& v\in\reg(Q)\\
                         \{v\}\coprod Q^1(v,-)& v\in\sing(Q).\end{matrix}\right.
\]
Because $\#X_v=\#X$ is infinite and $\#R_v\le \# X_v$, there is a bijection $\zeta_v:X_v\to R_v\times X$. Set $X_e=\zeta_{s(e)}^{-1}(\{e\}\times X_{s(e)})$
$(e\in Q^1)$. By construction, \eqref{decompo:regv} is satisfied. For $e\in Q^1$, let $r^{-1}\times 1:\{r(e)\}\times X\to \{e\}\times X$ be the obvious bijection. Define a semigroup homomorphism $S:\cS(Q)\to \cI(X)$ by setting
\[
S_v=1_{X_v},\quad S_e=\zeta^{-1}_{s(e)}(r^{-1}\times 1)\phi:X_{r(e)}\to X_{e}, \quad S_{e^*}=S_e^{-1}\qquad (v\in Q^0,\ \ e\in Q^1).
\]
Let $\cG$ be the groupoid of germs associated to this action of $S$ on $X$, and consider the nondegenerate spatial representation $\rho:L_Q\to\caL(\ell^p(\cG))$ of \eqref{map:laro}. Then $\rho$ is injective by Lemma \ref{lem:grpd}; furthermore, $\#\cG=\aleph_0$ and any bijection $\cG\cong \N$ induces a spatial isometric isomorphism $\ell^p(\cG)\cong\ell^p(\N)$. 
\end{proof}

\section{Matrix algebras and spatial representations}\label{sec:matspat}

Let $1\le n\le \infty$ and let $A$ be an algebra. Write $M_n$ for the algebra of $n\times n$-matrices with finitely many nonzero entries, and $M_nA=M_n\otimes A$. If $i,j\in\N$, we write $E_{i,j}$ for the canonical matrix unit. Let $Q$ be a countable graph, $X$ a $\sigma$-finite measure space, and $1\le p<\infty$. Call a representation $\rho:M_n(L_Q)\to \caL(L^p(X))$ \emph{spatial} if for every $x\in Q^0\cup Q^1$ and $i,j$, $\rho(E_{i,j}\otimes x)$ is a spatial partial isometry with reverse $\rho(E_{j,i}\otimes x^*)$.

\begin{rem}\label{rem:matspatleav}
Let $n\le \infty$ and let $M_nQ$ be the graph obtained by adding a head
\[
\xymatrix{\dots\ar[r]&v_i\ar[r]^{e^v_i}&v_{i-1}\ar[r]^{e^v_{i-1}}&\dots\ar[r]^{e^v_2}& v_1\ar[r]^{e^v_1}& v}
\]
for each $v\in Q^0$ and $i<n$. By \cite{abratom}*{Propositions 9.3 and 9.8}, there is a $*$-isomorphism 
\begin{gather}\label{map:LM=ML}
L_{M_nQ}\iso M_nL_Q,\\
v\mapsto E_{1,1}\otimes v, \quad v_i\mapsto E_{i+1,i+1}\otimes v\nonumber\\
e\mapsto E_{1,1}\otimes e, \quad e^v_i\mapsto E_{i+1,i}\otimes e\nonumber
\end{gather}
It is clear that a representation $M_nL_Q\to\caL(L^p(X))$ is spatial in the matricial sense above if and only if its composition with the map \eqref{map:LM=ML} is a spatial representation of $L_{M_nQ}$. 
\end{rem}

\begin{exa}\label{ex:rhoinfty}
 Let $\sigma:L_Q\to \caL(L^p(X))$ be a spatial representation. Let $I=\{1,\dots,n\}$ if $n$ is finite, and $I=\N$ if $n=\infty$. 
We have a canonical isometric isomorphism $L^p(I\times X)\cong \ell^p(I,L^p(X))$. Let 
\begin{gather*}
\sigma_I:M_nL_Q\to\caL(\ell^p(I,L^p(X)))\\
 \sigma_I(E_{i,j}\otimes a)(\xi)(k)=\delta_{k,i}\sigma(a)(\xi(j)). 
\end{gather*}
Then $\sigma_I$ is spatial. Indeed if $a\in \cS(Q)$ and $\sigma(a)$ is a spatial isometry with domain support $E$ and rank support $F$, then $\sigma_I(E_{i,j}\otimes a)$ is a spatial isometry with domain support $\{j\}\times E$ and range support $\{i\}\times F$. We remark that for $I=\{1,\dots,n\}$, $\sigma_I$ is the representation induced by the \emph{amplification} of $\sigma$ in the sense of \cite{gardelupi}*{Definition 4.10}.
\end{exa}

\begin{lem}\label{lem:matspat} 
Let $Q$ be a countable graph, $I$ a countable set, $X$ a $\sigma$-finite measure space, $1\le p<\infty$, and $\rho:M_IL_Q\to \caL(L^p(X))$ a nondegenerate spatial representation. Then there exist a $\sigma$-finite measure space $Y$, a spatial representation $\sigma:L_Q\to \caL(L^p(Y))$ and a spatial isometric isomorphism $u:\ell^p(I,L^p(Y))\to L^p(X)$ such that 
$\rho(a)=u\sigma_I(a)u^{-1}$ ($a\in L_Q$).
\end{lem}
\begin{proof}
Let $X_{i,v}$ be the domain support of the spatial idempotent $\rho(E_{i,i}\otimes v)$ ($\in I, v\in Q^0$). Set $X_i=\coprod_{v\in Q^0}X_{i,v}$; we have $X=\coprod_{i\in I}X_i$. Hence we have $L^p$-direct sum decompositions $L^p(X)=\bigoplus_{i\in I}L^p(X_i)$ and $L^p(X_i)=\bigoplus_{v\in Q^0}L^p(X_{i,v})$. Choose $i_0\in I$, and let $Y=X_{i_0}$. Then $u=\bigoplus_{i,v}\rho(E_{i,i_0}\otimes v)$ is a spatial isometric isomorphism $\ell^p(I,L^p(Y))=\bigoplus_{i\in I}L^p(Y)\to L^p(X)$. Let $\sigma:L_Q\to \caL(L^p(Y))$, $\sigma(a)=\rho(E_{i_0,i_0}\otimes a)$. One checks that $u$ conjugates $\sigma_I$ to $\rho$, concluding the proof. 
\end{proof}

\section{A spatiality criterion}\label{sec:crite}

We write $M_n=M_n\C$ for the matrix algebra and $M_\infty=\bigcup_nM_n$. We have a natural identification $M_n=\caL(\ell^p(\{1,\dots,n\})$ for $n<\infty$
and a natural embedding $M_\infty\to \caL(\ell^p(\N))$; by pulling back the operator norm, we get a norm $||\ \ ||_p$ on $M_n$ ($1\le n\le\infty$) which makes the latter into a normed algebra $M^p_n$. If $I$ is a set and 
\begin{equation}\label{uln}
\ul{n}=(n_i)_{i\in I}
\end{equation}
is a family with $1\le n_i\le \infty$, we write 
\begin{equation}\label{mnp}
M_{\ul{n}}^p=\bigoplus_{i\in I}M_{n_i}^p
\end{equation} 
for the algebraic direct sum equipped with the supremum norm $||(a_i)||=\sup_{i\in I}||a_i||_p$. We write $E^i_{a,b}$ $(i\in I)$, $1\le a,b\le n_i$ for the canonical matrix unit. 

The following proposition generalizes \cite{chris1}*{Theorem 7.2}.

\begin{prop}\label{prop:mat} Let $p\in [1,\infty)$, $p\ne 2$, $I$ a countable set, $\ul{n}$ as in \eqref{uln}, and $M_{\ul{n}}^p$ as in \eqref{mnp}. Let $X=(X,\cB,\mu)$ be a $\sigma$-finite measure space with $\mu\ne 0$. The following are equivalent for a nondegenerate representation $\rho:M_{\ul{n}}^p\to \caL(L^p(X))$.
\item[i)] $\rho(E^i_{a,b})$ is a spatial partial isometry for all $i\in I, 1\le a,b\le n_i$.
\item[ii)] $\rho$ is contractive.
\end{prop}
\begin{proof}
Assume that i) holds. Then each $\rho(E^i_{a,a})$ is a spatial idempotent, whence by Example \ref{exa:idempotent} there is $X^{i}_a\in\cB$ such that 
$\rho(E^i_{a,a})=\pi_{X_a^i}$ is the canonical projection. For each $i\in I$ put $\cN_i=\N$ if $n_i=\infty$ and $\cN_i=\{1,\dots,n_i\}$ if $n_i<\infty$. Because $\rho$ is nondegenerate, we have $X=\coprod_{i\in I}\coprod_{a\in\cN_i}X^i_a$. Put $X^i=\coprod_{a\in\cN_i}X^i_a$. By restriction, we obtain a nondegenerate representation $\rho^i:M_{n_i}\to \caL(L^p(X^i))$ satisfying i); hence we may assume that $I=\{1\}$ has only one element.  If $n<\infty$, nondegeneracy implies that $\rho(1)=1$, so $\rho$ is contractive by \cite{chris1}*{Theorem 7.2}. Assume $n=\infty$. Proceed as in loc. cit., using the partial isometries $\rho(E_{1,a}):L^p(X_a)\to L^p(X_1)$
to construct an isometry $u:L^p(X)\to \ell^p(\N,L^p(X_1))=\ell^p(\N)\otimes_pL^p(X_1)$ (the $L^p$-tensor product) that conjugates $\rho$ to the contractive representation $T\mapsto T\otimes 1$. It follows that $\rho$ is contractive, concluding the proof that i)$\Rightarrow$ii). Assume now that ii) holds. Then $\{\rho(E^i_{a,a}):i\in I, a\in\cN_i\}$ is a family of orthogonal idempotents. Let  $B^i_a=\rho(E^i_{a,a})L^p(X)$; then the algebraic direct sum $B=\bigoplus_{i,a}B^i_a$ is dense in $L^p(X)$. For each $z\in\mathbb{S}^1$, $i\in I$ and $a\in \cN_i$ define an operator $u_{i,a}(z):B\to B$ as multiplication by $z$ on $B^i_a$ and the identity on every other summand. Because $\rho$ is contractive, $u_{i,a}(z)$ has norm $1$, so it extends to a norm $1$ operator $u_{i,a}(z)\in\caL(L^p(X))$.
Since this also holds for $u_{i,a}(z^{-1})$, $u_{i,a}(z)$ is a bijective isometry. Hence it is spatial, by the Banach-Lamperti theorem. Now proceed as in \cite{chris1}*{page 42} to deduce that $\rho(E^i_{a,a})=(1-u_{i,a}(-1))/2$ is a spatial idempotent. Hence there exists $X_a^i\in\cB$ such that $B_a^i=L^p(X_a^i)$ and $X=\coprod X_a^i$. Since $\rho(E^i_{a,b})$ is an isometry $B^i_b\to B^i_a$, another application of the Banach-Lamperti theorem shows that it is spatial.
\end{proof}

Recall that the Leavitt path algebra is equipped with a $\Z$-grading $L_Q=\bigoplus_n(L_Q)_n$ where $(L_Q)_n$ is as in \eqref{grading}. Write 
$(L_Q)_{0,n}\subset (L_Q)_0$ for the subalgebra linearly spanned by the elements of the form $\alpha\beta^*$ with $r(\alpha)=r(\beta)$ and $|\alpha|=|\beta|\le n$. We have an increasing union
\[
(L_Q)_0=\bigcup_{n=0}^\infty (L_Q)_{0,n}.
\]
Each $(L_Q)_{0,n}$ is isomorphic to a direct sum of (possibly infinite dimensional) matrix algebras. 

\begin{thm}\label{thm:contraspat}(cf. \cite{gardelupi}*{Theorem 7.7} ) Let $X=(X,\cB,\mu)$ be a $\sigma$-finite measure space with $\mu\ne 0$, $p\in [1,\infty)$, $p\ne 2$, and $Q$ a countable graph. The following are equivalent for a nondegenerate representation $\rho:L_Q\to \caL(L^p(X))$.
\begin{itemize}
\item[i)] $\rho$ is spatial.
\item[ii)] $||\rho(e)||,||\rho(e^*)||\le 1$ ($e\in Q^1$) and the restriction of $\rho$ to $((L_Q)_{0,1},||\ \ ||_p)$ is contractive. 
\end{itemize}
\end{thm}
\begin{proof} The implication i)$\Rightarrow$ ii) is clear using Proposition \ref{prop:mat}. Assume that ii) holds; then $\rho(e)$ is a bijective isometry $\rho(r(e))L^p(X)\to \rho(ee^*)L^p(X)$ with inverse $\rho(e^*)$. By Proposition \ref{prop:mat}, $\rho(v)$ and $\rho(ee^*)$ are spatial idempotents $(v\in Q^0)$, $(e\in Q^1)$. Hence it follows from the Banach-Lamperti theorem \cite{chris1}*{Theorem 6.9} and from \cite{chris1}*{Lemma 6.15} that $\rho(e)$ and $\rho(e^*)$ are spatial. This concludes the proof. 
\end{proof}

\begin{rem}\label{rem:ojo} The assumption that $\rho$ be nondegenerate in necessary in both Proposition \ref{prop:mat} and Theorem \ref{thm:contraspat}. For example the trivial graph on one vertex has Leavitt algebra $\C$, which equals $M_1^p$ for all $1\le p<\infty$, and the representation $\C\to M^p_2$ that maps $1$ to the idempotent of Remark \ref{rem:idempotent} is contractive but not spatial. 
\end{rem}

\section{The $L^p$-operator algebra \(\cO^p(Q)\)}\label{sec:opq}

\begin{defi}\label{defi:lpopalg} Let $p\in [1,\infty)$. An \emph{$L^p$-operator algebra} is a Banach algebra $B$ together with a norm on each $M_nB$ that makes into a Banach algebra in such a way that there exists a nondegenerate representation $\rho:B\to\caL(L^p(X))$ for some $\sigma$-finite measure space $X$, such that $M_n\rho:M_nB\to M_n\caL(L^p(X))=\caL(L^p(\coprod_{i=1}^nX))$ is isometric for each $1\le n<\infty$. We call $B$ \emph{standard} if $X$ can be chosen to be a standard Borel space. A homomorphism $f:A\to B$ between $L^p$-operator algebras is \emph{$p$-completely contractive} (resp. \emph{isometric}) if $M_nf$ is contractive (resp. isometric) for every $n$. 
\end{defi}

\begin{rem}\label{rem:lpopalgstand} By \cite{chris2}*{Proposition 1.25}, any separable $L^p$-operator algebra admits an isometric representation in a separable, whence standard $L^p$-space.  Thus a separable $L^p$-operator algebra is automatically standard. 
\end{rem}

\begin{rem}\label{rem:lpopalgnondeg} If $p\ne 1$ and $B$ has a contractive approximate unit, then the condition that the isometric representation in Definition \ref{defi:lpopalg} be nondegenerate can be dropped, by \cite{gardethiel2}*{Theorem 3.19}. 
\end{rem}

A \emph{spatial} $p$-seminorm is a seminorm $h:L_Q\to\R_{\ge 0}$ such that there exist
a $\sigma$-finite measure space $X$ and spatial representation $\rho:L_Q\to\caL(L^p(X))$ such that $h(a)=||\rho(a)||$ ($a\in L_Q$). Observe that by Lemma \ref{lem:nondegenerate}, every spatial seminorm is induced by a nondegenerate spatial representation. Put
\begin{align}
\pssn(Q)=&\{h:L_Q\to\R_{\ge 0} \text{ spatial $p$-seminorm}\},\label{pssn}\\
||a||=&\sup\{h(a): h\in\pssn(Q) \}\label{spatnorm}.
\end{align}
By Proposition \ref{prop:sigmarep}, $||\ \ ||$ is a norm. Write  $\cO^p(Q)=\overline{L_Q}^{||\ \ ||}$ for the completion of $L_Q$ with respect to the norm \eqref{spatnorm}; $\cO^p(Q)$ is a Banach algebra, and the canonical map $L_Q\to \cO^p(Q)$ is injective, again by Proposition \ref{prop:sigmarep}. Since $Q$ is countable, there is a countable family $\{\rho_n\}$ of $\sigma$-finite nondegenerate spatial representations such that $||\ \ ||$ is the norm associated to the $L^p$-direct sum 
\begin{equation}\label{map:rhogood}
\rho=\bigoplus_n\rho_n:L_Q\to \caL(L^p(\coprod_nX_n))
\end{equation}
which is a nondegenerate spatial representation. Hence $\cO^p(Q)$ is isometrically isomorphic to the closure of $\rho(L_Q)$.

\begin{prop}\label{prop:matopq}
Let $Q$ be a countable graph. Then $\cO^p(Q)$ has a canonical structure of $L^p$-operator algebra such that there is an isometric isomorphism 
$M_n\cO^p(Q)\cong \cO^p(M_nQ)$ ($\infty>n\ge 1$). 
\end{prop}
\begin{proof} By Remark \ref{rem:matspatleav}, the canonical map $M_nL_Q\cong L_{M_nQ}\to \cO^p(M_nQ)$ is universal for $L^p$- spatial representations. By Lemma \ref{lem:matspat} and the discussion above, for each $n$ there is a spatial representation $\rho_n:L_Q\to\caL(L^p(X_n))$ such that $||\ \ ||_n:=||M_n\rho_n(\ \ )||$ is the supremum of all $p$-spatial norms on $L_{M_n(Q)}$. Let $X=\coprod_nX_n$ and let $\rho=\bigoplus_n\rho_n:L_Q\to \caL(L^p(X))$ be the 
$L^p$-direct sum. Then $||M_n\rho(\ \ )||=||\ \ ||_n$ for all $n\ge 1$, and we have isometric isomorphisms
 \[
\cO^p(M_nQ)\cong\ol{M_n\rho(M_n(L_Q))}=\ol{M_n(\rho(L_Q))}=M_n(\ol{\rho(L_Q)})\cong M_n\cO^p(Q).
\]
\end{proof}

\begin{thm}\label{thm:contraspat15}
Let $X$ be $\sigma$-finite measure space with nonzero measure, $p\in [1,\infty)$, $p\ne 2$, $Q$ a countable graph, $\hat{\rho}:\cO^p(Q)\to \caL(L^p(X))$ a nondegenerate representation and $\rho:L_Q\to \caL(L^p(X))$ the restriction of $\hat{\rho}$. Then the following conditions are equivalent:
\begin{itemize}
\item[i)]  $\rho$ is spatial.
\item[ii)] $\hat{\rho}$ is contractive. 
\end{itemize}
\end{thm}
\begin{proof}
If $\rho$ is spatial then it induces a contractive homomorphism $\hat{\rho}':\cO^p(Q)\to\caL(L^p(X))$ which agrees with $\rho$ on $L_Q$; since $\hat{\rho}$ does the same, we must have $\hat{\rho}=\hat{\rho}'$. This proves that i)$\Rightarrow$ii). Conversely if $ii)$ holds, then $\rho$ is spatial by Theorem  
\ref{thm:contraspat}.
\end{proof}

\begin{thm}\label{thm:contraspat2}
Let $X$ be $\sigma$-finite measure space with nonzero measure, $p$ as in Theorem \ref{thm:contraspat15}, $Q$ a countable graph, $\hat{\rho}:\cO^p(Q)\to \caL(L^p(X))$ a representation and $\rho:L_Q\to \caL(L^p(X))$ the restriction of $\hat{\rho}$. Further assume either that $p\ne 1$ or that $Q^0$ is finite. Then the following conditions are equivalent:
\item[i)] There exist a $\sigma$-finite measure space $Y$, an isometry $\iota:L^p(Y)\to L^p(X)$, a norm $1$ operator $\pi:L^p(X)\to L^p(Y)$ such that $\pi\iota=1$, and a spatial representation $\rho':L_Q\to\caL(L^p(Y))$, such that for $f:\caL(L^p(Y))\to\caL(L^p(X))$, $f(T)=\iota T\pi$, the following diagram commutes
\[
\xymatrix{L_Q\ar[dr]_{\rho'}\ar[rr]^\rho&&\caL(L^p(X))\\
               &\caL(L^p(Y))\ar[ur]_f&}
\]
\item[ii)] $\hat{\rho}$ is contractive. 
\end{thm}
\begin{proof}
If i) is satisfied, then $\rho'$ factors through a contractive representation $\hat{\rho'}:\cO^p(Q)\to\caL(L^p(Y))$. Thus $f\hat{\rho'}=\hat{\rho}$ is contractive. Assume conversely that ii) holds. Let $E\subset\caL(L^p(X))$ be the closure of $\rho(L_Q)\caL(L^p(X))$. If $Q^0$ is finite then $\cO^p(Q)$ is unital with unit $1=\sum_{v\in Q^0}v$ which has norm $1$; thus $E$ is the image of the contractive idempotent $\rho(1)$. For general $Q$, the family $\{\sum_{v\in F}v\}$ indexed by the finite subsets of $Q^0$ is a contractive approximate unit of $\cO^p(Q)$; hence if $p\ne 1$, then again $E$ is the image of a contractive idempotent, by \cite{gardethiel2}*{Corollary 3.13}. Hence under either hypothesis, by  
\cite{ando}*{Theorem 4} there are a contractive projection $\pi':L^p(X)\to E$ and  an isometric isomorphism $h:E\to L^p(Y)$ for some standard Borel space $Y$. If $p=1$ but  Put $\pi=h\pi'$, let $\iota$ be $h^{-1}$ followed by the inclusion $E\subset L^p(X)$, and set $\rho'(a)=h\rho(a)h^{-1}$. It is clear that the diagram commutes; moreover, $\rho'$ is spatial by Theorem \ref{thm:contraspat}.  
\end{proof}

\begin{rem}\label{rem:gardelupi}
Let $\cS(Q)$ be the semigroup of \eqref{semiQ} and let $1<p<\infty$. Let $F^p_\tight(\cS(Q))$ be the standard $L^p$-operator algebra of 
\cite{gardelupi}*{Definition 6.7}; $F^p_\tight(\cS(Q))$ is universal for tight $L^p$-representations of $\cS(Q)$ which are spatial in the sense of \cite{gardelupi}*{Definition 4.6} and take values in $L^p$-spaces of standard Borel spaces. As pointed out above, the spatiality notion of \cite{gardelupi} agrees with ours for $p\ne 2$. Hence by Lemma \ref{lem:tight} and the universal property of $\cO^p(Q)$, for $p\in (1,\infty)$, $p\ne 2$, we have a canonical contractive homomorphism
\begin{equation}\label{map:opfp}
\cO^p(Q)\to F^p_\tight(S(Q))\qquad (p\in (1,\infty),\quad p\ne 2).
\end{equation}
Moreover, since the $p$-operator space structure on $F^p_\tight(S(Q))$ is defined in \cite{gardelupi} so that $M_n(F^p_\tight(S(Q))=F^p_\tight(S(M_nQ))$, the induced map $M_n(\cO^p(Q))\to M_n(F^p_\tight(S(Q)))$ is also contractive, by Proposition \ref{prop:matopq}. In other words \eqref{map:opfp} is $p$-completely contractive.
\end{rem}

\begin{prop}\label{prop:opq=fps}
The map \eqref{map:opfp} is a $p$-completely isometric isomorphism.
\end{prop}
\begin{proof}
It suffices to show that $F_\tight^p(\cS(Q))$ is universal for all $\sigma$-finite representations. Let $X$ be a $\sigma$-finite measure space and let $\rho:L_Q\to\caL(L^p(X))$ be a spatial representation; we have to show that $\rho$ factors through $L_Q\to F_\tight^p(\cS(Q))$. An argument similar to that of the proof of Theorem \ref{thm:contraspat2} shows that $\rho$ factors through a nondegenerate representation $\rho':L_Q\to \caL(L^p(Y))$ with $Y$ standard Borel. Thus $\rho$ factors through $L_Q\to F^p_\tight(\cS(Q))$, as required.
\end{proof}

\section{Spatial seminorms, desingularization, and source removal}

Let $Q$ be a countable, singular graph. Recall from \cite{Abr-Pino2}*{Section 5} that the \emph{desingularization} of $Q$ is a nonsingular graph $Q_\d$ obtained from $Q$ as follows. For each sink $v$, add an infinite tail
\begin{equation}\label{ftail}
v=v_0\overset{f_1}\to v_1\overset{f_2}\to v_2\overset{f_3}\to\cdots
\end{equation} 
For each infinite emitter $v$, number the elements of $s^{-1}(v)=\{e_1,e_2,\dots\}$ and add a tail \eqref{ftail} and an arrow $g_i:v_i\to r(e_i)$ 
$(1\le i)$. There is a canonical $*$-monomorphism \cite{Abr-Pino2}*{Proposition 5.5}
\begin{gather}
\phi_\d:L_Q\to L_{Q_\d},\label{map:qqd}\\
\phi_\d(v)=v,\qquad
\phi_\d(e)=\left\{\begin{matrix} e & s(e)\in\reg(Q)\\ f_1\cdots f_ig_i& e=e_i\end{matrix}\right.\nonumber
\end{gather}
If $Q$ is a graph such that $\sour(Q)\ne\emptyset$, we may embed it in the source-free graph $Q_\o$ obtained by adding an infinite head
\begin{equation}\label{rtail}
w=w_0\overset{f_1}\leftarrow w_1\overset{f_2}\leftarrow w_2\overset{f_3}\leftarrow\cdots
\end{equation}
at each $w\in \sour(Q)$. The obvious inclusion $Q\subset Q_\o$ induces an algebra monomorphism
\begin{equation}\label{map:qqr}
\phi_\o:L_Q\to L_{Q_\o}.
\end{equation} 
Observe that for $\#\in\{\d,\o\}$, composition with $\phi_\#$ sends spatial representations of $L_{Q_\#}$ to spatial representations of $L_Q$. Hence we have an induced map
\begin{equation}\label{map:seminorms}
\pssn(Q_\#)\to \pssn(Q),\quad h\mapsto h\circ \phi_\#.
\end{equation}

\begin{prop}\label{prop:seminorms}
Let $Q$ be a countable graph, $1\le p<\infty$, and $\#\in\{\o,\d\}$. Then the map \eqref{map:seminorms} is surjective. 
\end{prop}
\begin{proof} It suffices to show that for every nonzero spatial representation $\rho:L_Q\to \caL(L^p(X))$ there exist a spatial representation $\rho_\#:L_{Q_\#}\to \caL(L^p(Y))$ and a spatial isometry $s:L^p(X)\to L^p(Y)$ with reverse $t$, with both $Y$ and $s$ depending on $\rho$ and $\#$, such that for the map
$\sigma:\caL(L^p(X))\to \caL(L^p(Y))$, $\sigma(A)= sAt$, the following diagram commutes:
\begin{equation}\label{diag:seminorms}
\xymatrix{L_Q\ar[d]_{\phi_\#}\ar[r]^\rho& \caL(L^p(X))\ar[d]^{\sigma}\\
         L_{Q_\#}\ar[r]_{\rho_\#}& \caL(L^p(Y))\\
}
\end{equation}
We begin by the case $\#=\o$. If $\alpha\in\cP(Q)$, we write $X_\alpha$ for the support of the spatial projection $\rho(\alpha\alpha^*)$. Regard $\N$ as a measure space with counting measure; set $Y:=X\sqcup\bigsqcup_{w\in\sour(Q)} (X_w \times\mathbb{N})$. Let $s$ and $t$ be the inverse isometries induced by the inclusion $X\subset Y$. The canonical identification $X_w\to X_w\times\{n\}$ induces an isometric spatial isomorphism 
$\tau_n:L^p(X_w)\rightarrow L^p(X_w\times\{n\})$. Extend $\rho$ along $\phi_\o$ to a map $\rho_\o:L_{Q_\o}\to \caL(L^p(Y))$ by setting 
$\rho_\o(w_n):=Id_{L^p(X_w\times\{n\})},$ $\rho_\o(f_n):=\tau_{n}\tau_{n-1}^{-1}$, $\rho_\o(f_n^*)=\tau_{n-1}\tau_{n}^{-1}$.
One checks that $\rho_\o$ is well-defined and makes \eqref{diag:seminorms} commute. Next we consider the case $\#=\d$. The measure space $Y$  will be a coproduct 
$$Y=X\sqcup\coprod_{v\in \sing(Q),n\ge 1}Y_{v_n};$$
the isometries $s,t$ will be those induced by the inclusion $X\subset Y$. 
For $v\in\sink(Q)$, we set $Y_{v_n}=X_v\times \{n\}$, $\tau_n:X_v\iso X_v\times \{n\}$ the obvious bijection, and put $\rho_\d(f_n)=\tau_{n-1}\tau_n^{-1}$. If $v\in \infem(Q)$ and $X'_v=X_v\setminus\coprod_{i=1}^\infty X_{e_i}$, we set 
$$Y_{v_n}=X'_v\sqcup\coprod_{i\ge n}X_{e_i}$$
and let $\rho_\d(f_n)$ be induced by the inclusion $Y_{v_n}\subset Y_{v_{n-1}}$ and $\rho_\d(g_n)$ by the composite of 
$\rho(e_n):L^p(X_{r(e_n)})\to L^p(X_{e_n})$ followed by the inclusion $L^p(X_{e_n})\subset L^p(Y_{v_n})$. One checks that this prescription defines a spatial representation $\rho_\d:L_{Q_\d}\to \caL(L^p(Y))$ that makes \eqref{diag:seminorms} commute.\end{proof}
\begin{coro}\label{coro:seminorms}
The canonical homomorphisms \eqref{map:qqd} and \eqref{map:qqr} induce isometric homomorphisms $\opq\to \cO^p(Q_\d)$ and $\opq\to\cO^p(Q_\o)$.
\end{coro}

\section{A uniqueness theorem}\label{sec_O(Q)}

The purpose of this section is to prove the following theorem.

\begin{theorem}\label{thm:indep}
Let $Q$ be a countable graph and $1\le p<\infty$. If $L_Q$ is simple, then the set $\pssn(Q)$ of $p$-spatial seminorms on $L_Q$ has only one nonzero element. In particular, if $\rho:L_Q\to \caL(L^p(X))$ is any nonzero spatial representation and $\overline{\rho(L_Q)}\subset \caL(L^p(X))$ the operator norm closure, then the natural map is an isometric isomorphism 
$$\mathcal{O}^p(Q)\iso \overline{\rho(L_Q)}.$$ 
\end{theorem}

The proof of Theorem \ref{thm:indep} will be given at the end of the section, after a series of propositions, definitions, and lemmas, which adapt and extend those in \cite{chris1}*{Section ~8}.

\begin{definition}\label{def_free-approx.}
 Let $Q$ be a countable row-finite graph, $p\in [1,\infty)$, $X=(X,\cB,\mu)$ a $\sigma$-finite measure space, and  
$\rho: L_Q\rightarrow \caL(L^p(X))$ a representation.

\begin{enumerate}
  \item We say that $\rho$ is \emph{free} if there is a partition $X=\bigsqcup_{m\in\mathbb{Z}}E_m$, $E_m\in\cB$, such that for all $m\in\mathbb{Z}$, $e\in Q^1$, we have
      \begin{equation}\label{libre}
        \rho(e)(L^p(E_m))\subset L^p(E_{m+1})\ \ \ \text{and}\ \ \ \rho(e^*)(L^p(E_m))\subset L^p(E_{m-1}).
      \end{equation}

  \item We say that $\rho$ is \emph{approximately free} if for every  $N\in\mathbb{N}$, there are $n\geq N$ and a partition $X=\displaystyle{\bigsqcup_{m=0}^{n-1}}E_m$, $E_m\in\cB$, such that for $m=0,\ldots,n-1$ and all $e\in Q^1$ \eqref{libre} holds if we set $E_n=E_0$ and $E_{-1}=E_{n-1}$.
\end{enumerate}
\end{definition}

\begin{lemma}\label{lem:rho^u}
Let $p\geq 1$, $X=(X,\cB,\mu)$ and $Y=(Y,\cC,\nu)$ $\sigma$-finite measure spaces, $Q$ a row-finite graph, $\rho:L_Q\rightarrow \caL(L^p(X))$ a representation, and $u\in \mathcal{L}(L^p(Y))$ an invertible operator. Then, there is unique representation $\rho^u:L_Q\rightarrow \caL(L^p(X\times Y))$ such that, for all $e\in Q^1$, we have $\rho^u(e)=\rho(e)\otimes u$ and $\rho^u(e^*)=\rho(e^*)\otimes u^{-1}$.

Moreover, $\rho^u$ has the following properties:

\begin{itemize}
              \item [(a)] If $\alpha\in L_Q$ is homogeneous of degree $k$ with respect to the $\mathbb{Z}$-grading of \eqref{grading}, then $\rho^u(\alpha)=\rho(\alpha)\otimes u^k$.
              \item [(b)] If $u$ is isometric, $p\neq 2$ and $\rho$ is spatial, then $\rho^u$ is spatial.
              \item [(c)] If there is a partition $Y=\displaystyle{\coprod_{m\in\mathbb{Z}}}F_m$, $F_m\in\cC$, such that $u(L^p(F_m))=L^p(F_{m+1})\ \forall m\in\mathbb{Z}$, then $\rho^u$ is free in the sense of Definition \ref{def_free-approx.}.
            \end{itemize}
\end{lemma}

\begin{proof}
The proof is analogous to that of \cite{chris1}*{Lemma 8.2} using Lemma \ref{lem_exist_repres} instead of \cite{chris1}*{Lemmas 2.18, 2.19 and 2.20}.
\end{proof}

\begin{proposition}\label{prop:ushift}
Let $p$, $X$, $Q$, and $\rho$ be as in Lemma \ref{lem:rho^u}. Let $u\in \mathcal{L}(\ell^p(\mathbb{Z}))$ be the shift operator, $(u(x))(m):=x(m-1)$ ($x\in \ell^p(\mathbb{Z})$). Let $\rho^u$ be as in Lemma \ref{lem:rho^u}. Then, for all $a\in L_Q$, we have $\|\rho^u(a)\|\geq\|\rho(a)\|$.
\end{proposition}
\begin{proof}
The proof is analogous to that of \cite{chris1}*{Proposition 8.3}, using Lemma \ref{lem:rho^u} instead of \cite{chris1}*{Lemma 8.2}.
\end{proof}

\begin{lemma}\label{lem:disj}
Let $Q$ be a nonsingular countable graph such that $L_Q$ is simple. Let $X=(X,\cB,\mu)$ be a $\sigma$-finite measure space.
Let $\{X_v\}_{v\in Q^0}\subset\cB$ a family of sets of nonzero measure,  $\{X_e\}_{e\in Q^1}\subset\cB$  a disjoint family such that $X=\displaystyle{\coprod_{v\in Q^0}} X_v$ and $X_v=\displaystyle{\coprod_{\{e:s(e)=v\}}}X_e\ \ (\forall v\in Q^0)$, and $$S_e: (X_{r(e)},\mathcal{B}_{|_{X_{r(e)}}},\mu_{|_{X_{r(e)}}})\rightarrow (X_{e},\mathcal{B}_{|_{X_{e}}},\mu_{|_{X_{e}}})\qquad (e\in Q^1)$$ a bijective measurable set transformation. If $\alpha=\alpha_1\cdots\alpha_m$ is a path, write $S_\alpha=S_{\alpha_1}\circ\cdots\circ S_{\alpha_m}$. 
Then, for each $n\ge 0$ and each $v\in Q^0$ there is a set $E_v\in\cB_{|X_v}$ such that $\mu(E_v)\neq0$, and such that the family 
$$\{S_\alpha(E_v):r(\alpha)=v, |\alpha|\le n\}$$
is disjoint. 
\end{lemma}

\begin{proof} 
We shall use the fact that, because $L_Q$ is simple, $Q$ is cofinal, i.e. for every $v\in Q^0$ and each cycle $c$ there is a path starting at $v$ and ending at some vertex in $c$ (see \cite{libro_Pere}*{Theorem 2.9.7}). Let $v\in Q^0$. If $v\in Q^0$ is not in any cycle, we set $E_v=X_v$; observe that $\mu(E_v)\neq 0$ by hypothesis. Because $v$ is not in any cycle, any two distinct paths ending in $v$ are incomparable, and so $E_v$ satisfies the disjointness condition of the lemma. Next assume that $v$ belongs to a cycle. Let $\alpha:=\alpha_v$ be a cycle based at $v$ and let $\beta$ be a closed path with $s(\beta)=v$ that agrees with $\alpha$ up to an exit, goes out following the exit, returns to $c$ (which is possible by cofinality) and follows it till it gets back to $v$.
Consider the infinite path 
\[
\gamma:=\alpha\beta\alpha\alpha\beta\beta\alpha\alpha\alpha\beta\beta\beta\ldots
\] 
It is long, but straightforward to check that 
\begin{equation}\label{repe}
    \nexists\theta\in\cP(Q) \text{ such that } \theta\theta\geq\gamma.
\end{equation}
Let $n\in\mathbb{N}$ and $v\in E^0$. For $i\ge 1$, let $\gamma_i$ be the $i$-th edge of $\gamma$.  Put 
\[
\cB\owns E_v:=X_{\gamma_1\ldots\gamma_{2n}}.
\] 
Then $\mu(E_v)\neq0$ because $\mu(X_{w})\ne 0$ for all $w\in Q^0$. Let $\eta$ and $\tau$ be different paths such that $r(\eta)=r(\tau)=v$, of lengths $k$ and $l$ respectively ($k\le l\leq n$). We have to check that $S_{\eta}(E_v)$ and $S_{\tau}(E_v)$ are disjoint. If  $k=0$ this is clear from Remark \ref{rem:comparables}, because $S_{\tau}(E_v)=X_{\tau\gamma_1\ldots\gamma_{2n}}$ and the paths $\tau\gamma_1\ldots\gamma_{2n}$ and $\gamma_1\ldots\gamma_{2n}$ are incomparable, by \eqref{repe}. So assume that $0<k\leq l$; if $\eta$ and $\tau$ are incomparable, we are done. Otherwise, we  must have $\eta\geqslant\tau$; say $\tau=\eta\delta$. Hence $S_{\eta}(E_v)\cap S_{\tau}(E_v)=S_{\eta}(E_v\cap S_{\delta}(E_v))$ has measure zero because $E_v\cap S_{\delta}(E_v)$ does.
\end{proof}

 Let $(X,\mathcal{B},\mu)$ be a $\sigma$-finite measure space and $\tau_1,\ldots,\tau_n\in\mathcal{L}(L^p(X))$ spatial partial isometries with reverses $\sigma_1,\ldots,\sigma_n$. Call $\tau_1,\ldots,\tau_n$ \emph{orthogonal} if $\tau_j\sigma_i=\sigma_i\tau_j=0\ \forall i\neq j$.

\begin{lemma}\label{lem:orthogonal}
 Let $X$ be a $\sigma$-finite measure space, $p\in[1,\infty)$, $\tau_1,\ldots,\tau_n\in\caL(L^p(X))$ orthogonal spatial partial isometries,
$\lambda\in\C^n$, and $\tau_\lambda=\displaystyle{\sum_{i=1}^n}\lambda_i\tau_i$. Then $\|\tau_\lambda\|=\|\lambda\|_\infty$.
\end{lemma}
\begin{proof} Straightforward.
\end{proof}

\begin{prop}\label{prop:larga} Let $Q$ be a nonsingular countable graph without sources. Let $p\in [1,\infty)$, $p\ne 2$, and let $X$ and $Y$ be measure spaces and $\rho:L_Q\to \caL(L^p(X))$ and $\phi:L_Q\to \caL(L^p(Y))$ spatial representations. Assume that $L_Q$ is simple and that $\rho$ is approximately free.
Then $$\|\rho(a)\|\le\|\phi(a)\|\quad (a\in L_Q).$$
\end{prop}
\begin{proof} This proposition generalizes \cite{chris1}*{Proposition 8.6}; we shall adapt the argument therein using Lemma \ref{lem:disj} instead of \cite{chris1}*{Lemma 8.5}.
Let $X'=\coprod_{v\in Q^0}X_v$; observe that the corestriction $\rho'$ of $\rho$ to $\caL(L^p(X'))$ is approximately free. Hence by Lemma \ref{lem:nondegenerate} we may assume that $\rho$ and $\phi$ are both nondegenerate. For each $\alpha\in \cP=\cP(Q)$, let $R_\alpha$ and $S_\alpha$ be the bijective measurable set transformations $X_{r(\alpha)}\to X_\alpha$, $Y_{r(\alpha)}\to Y_\alpha$ associated to $\rho(\alpha)$ and $\phi(\alpha)$, as in Remark \ref{rem:comparables}. We have to show that if $a\in L_Q$ is such that $\|\rho(a)\|=1$, then $\|\phi(a)\|\ge 1$. By Lemma \ref{lem:escr_sum}, there are $N_0\ge 0$, a finite set $F_0\subset\cP$ and a finitely supported function $\lambda^0:F_0\times \cP_{N_0}\to \C$
such that
\[
a=\sum_{\alpha\in F_0}\sum_{\beta\in\cP_{N_0}}\lambda^0_{\alpha,\beta}\alpha\beta^*
\]
Because $\sour(Q)=\emptyset$ by hypothesis, for each $v\in s(F_0)$ we may choose a path $\tau_v\in \cP_{N_0}$ with $r(\tau_v)=v$. Put 
\[
x=\sum_{v\in s(F_0)}\tau_{v},\quad b=xa.
\]
Because every path in the set $\tau_{F_0}=\{\tau_v:v\in s(F_0)\}$ is of length $N_0$, any two of them are incomparable. Hence by Remark \ref{rem:comparables}, the elements of $\rho(\tau_{F_0})$ are orthogonal spatial partial isometries. Therefore $\|\rho(x)\|=1$,
by Lemma \ref{lem:orthogonal}; similarly, $\|\rho(x^*)\|=1$. Hence 
$\|\rho(b)\|=\|\rho(a)\|=1$ and by the same argument, $\|\phi(b)\|=\|\phi(a)\|$. Therefore it suffices to show that for every $\epsilon>0$,
\begin{equation}\label{toshow}
\|\phi(b)\|>1-\epsilon.
\end{equation}

For $\beta\in\cP_{N_0}$ and $\alpha\in F_0$, let 
\[
\lambda_{\tau_{s(\alpha)}\alpha,\beta}=\lambda^0_{\alpha,\beta}.
\] 
Put $F=\{\tau_{s(\alpha)}\alpha:\alpha\in F_0\}$; the map $F_0\to F$, $\alpha\mapsto \tau_{s(\alpha)}\alpha$ is clearly surjective. Moreover, because $\tau_v\in \cP_{N_0}$ for all $v\in s(F_0)$, it is also injective. Using this in the third step, we obtain 

\begin{align*}
b=&(\sum_{v\in s(F_0)}\tau_v)(\sum_{\alpha\in F_0}\sum_{\beta\in\cP_{N_0}}\lambda^0_{\alpha,\beta}\alpha\beta^*)\\
=&\sum_{v\in s(F_0)}\sum_{\underset{s(\alpha)=v}{\alpha\in F_0}}\sum_{\beta\in\cP_{N_0}}\lambda_{\tau_v\alpha,\beta}\alpha\beta^*\\
=&\sum_{\alpha\in F}\sum_{\beta\in\cP_{N_0}}\lambda_{\alpha,\beta}\alpha\beta^*
\end{align*}

Let $N_1=\max\{|\alpha|:\alpha\in F_0\}$; then $N_0\le |\alpha|\le N_0+N_1$ for all $\alpha\in F$. If $N_0=N_1=0$, then $b$ is a linear combination of vertices, $b=\sum_v\lambda_v v$, whence by Lemma \ref{lem:orthogonal} we have
\[
\|\phi(b)\|=\|\lambda\|_\infty=\|\rho(b)\|=1.
\]
Hence \eqref{toshow} holds in this case. So we may assume $N_0+N_1>0$, and take $\ru>(N_0+N_1)(2/\epsilon)^p$. By our hypothesis on $\rho$, there are $N\ge \ru(N_0+N_1)$ and a partition 
\begin{equation}\label{decompo:aprofree}
X=\coprod_{n=0}^{N-1}D_n
\end{equation} 
such that for the remainder $\bar{n}$ of $n$ modulo $N$, we have $\rho(e)(L^p(D_{\bar{n}}))\subset L^p(D_{\ol{n+1}})$ and $\rho(e^*)(L^p(D_{\bar{n}}))\subset L^p(D_{\ol{n-1}})$. By the argument of  \cite{chris1}*{pages 54--55}, upon cyclic permutation of the $D_n$ if necessary, there exists 
\[
\xi=\sum_{m=0}^{N-1}\xi_m\in\bigoplus_{m=0}^{N-1}L^p(D_m)=L^p(X)
\]
with $\xi_m=0$ for $m\le N_0-1$ and for $m\ge N-N_1$, and such that $\|\xi\|\le 1$ and $\|\rho(b)\xi\|>1-\epsilon$. For each $\gamma\in\cP$, put 
$$
D_\gamma=R_\gamma(X_{r(\gamma)}\cap D_0)=D_{|\gamma|}\cap X_\gamma.
$$
Because $Q$ is nonsingular by hypothesis, and because we have assumed that $\rho$ is nondegenerate, for each $l\ge 0$ we have a decomposition \eqref{decompo:nondegreg}. It follows from this that
\begin{equation}\label{decompo:dmdgamma}
D_m=\coprod_{|\gamma|=m}D_\gamma\qquad (0\le m\le N-1).
\end{equation}
Let 
\[
W=\cP_{\le N-1}=\coprod_{0\le l\le N-1}\cP_l.
\]
It follows from \eqref{decompo:aprofree} and \eqref{decompo:dmdgamma} that
\[
X=\coprod_{\gamma\in W}D_\gamma.
\]
Hence we can write any $\eta\in L^p(X)$ as a sum
\[
\eta=\sum_{\gamma\in W}\eta_\gamma\qquad (\eta_\gamma\in L^p(D_\gamma)). 
\]
Next, by Lemma \ref{lem:disj}, for each $v\in Q^0$ there is a measurable set $E_v\subset Y$ of nonzero measure such that the family $\{S_\gamma(E_{r(\gamma)}):\gamma\in W\}$ is disjoint. Choose a norm-one element $\zeta_v\in L^p(E_v)$ for each $v\in Q^0$. Let
\[
u:L^p(X)\to L^p(X\times Y),\qquad u\eta=\sum_{\gamma\in W}\rho(\gamma)\eta_\gamma\otimes\phi(\gamma)\zeta_{r(\gamma)}.
\]
One checks, as in the proof of \cite{chris1}*{Proposition 8.6}, that $u$ is an isometry. Let $\psi=1\otimes \phi:L_Q\to \caL(L^p(X\times Y))$, 
be as in Lemma \ref{lem:rho^u}. Observe that 
\begin{equation}\label{phi=psi}
\|\psi(b)\|=\|\phi(b)\|.
\end{equation}
A calculation similar to that of the proof of \cite{chris1}*{Proposition 8.6} shows that for $\xi$ as above, 
\begin{equation}\label{intertwine}
u\rho(b)\xi=\psi(b)u\xi. 
\end{equation}
It follows from \eqref{phi=psi} and \eqref{intertwine} that \eqref{toshow} holds. This completes the proof.
\end{proof}

\noindent{\it Proof of Theorem \ref{thm:indep}} Because $L_Q$ is simple by hypothesis, the $C^*$-algebra $C^*(Q)$ is simple; thus every nonzero $*$-representation $L_Q\to \caL(L^2(X))$ induces the same norm. But by Remark \ref{rem:spat*} every spatial representation is a $*$-representation, so the theorem is clear for $p=2$. Assume $p\ne 2$. By Proposition \ref{prop:seminorms} and Corollary \ref{coro:seminorms}, we may assume that 
$Q$ is nonsingular and has no sources. By Lemma \ref{lem:rho^u} and Propositions \ref{prop:ushift} and \ref{prop:larga}, every spatial seminorm is associated to a free spatial representation. Applying Proposition \ref{prop:larga} again, we get that any two nonzero approximately free spatial representations lead to the same seminorm.\qed

\section{A simplicity theorem}\label{sec:simplicity}

\begin{thm}\label{thm:simple}
Let $p\in [1,\infty)$, $p\ne 2$. The following are equivalent for a countable graph $Q$.
\item[i)] $L_Q$ is simple. 
\item[ii)] Every spatial nonzero $L^p$-representation of $L_Q$ is injective.
\item[ii')] Every spatial nonzero representation $L_Q\to\caL(\ell^p(\N))$ is injective.
\item[ii")] Every nondegenerate spatial nonzero representation $L_Q\to\caL(\ell^p(\N))$ is injective.
\item[iii)] Every nondegenerate, contractive, nonzero $L^p$-representation of $\cO^p(Q)$ is injective.
\item[iii')] Every nondegenerate, contractive, nonzero representation $\cO^p(Q)\to\caL(\ell^p(\N))$ is injective. 
\goodbreak

If in addition we assume either that $p\ne 1$ or that $Q^0$ is finite, then the above conditions are also equivalent to the following.

\item[iv)] Every nonzero contractive homomorphism from $\cO^p(Q)$ to another $L^p$-operator algebra is injective.
\end{thm}

\begin{proof} If either $p\ne 1$ or $Q^0$ is finite, then iii) and iv) are equivalent, by Theorem \ref{thm:contraspat2}. Let $2\ne p\in [1,\infty)$. It follows from Theorems \ref{thm:contraspat15} and \ref{thm:indep} that i)$\Rightarrow$iii). By Lemma \ref{lem:nondegenerate} and Theorem \ref{thm:contraspat15}, iii)$\Rightarrow$ii). Similarly, iii')$\Rightarrow$ii"). It is clear that ii)$\Rightarrow$ii')$\Rightarrow$ii") and that iii)$\Rightarrow$iii'). It remains to show that ii")$\Rightarrow$i). By \cite{libro_Pere}*{Theorem 2.9.1}, $L_Q$ is simple if and only if $Q^0$ is the only nonempty hereditary and saturated subset of vertices, and every cycle in $Q$ has an exit. We shall show that if any of these two conditions does not hold, then ii") does not hold either. So suppose there is a proper hereditary and saturated subset $H\subset Q^0$. Let $Q/H$ be the quotient graph of \cite{libro_Pere}*{Definition 2.4.11}. Then the natural map $\pi:L_Q\to L_{Q/H}$ is a nonzero surjection with nonzero kernel the ideal $I(H)$ generated by $H$. Hence if $\rho$ is an injective nondegenerate spatial representation $L_{Q/H}\to \caL(\ell^p(\N))$ (which exists by Proposition \ref{prop:sigmarep}) then $\rho\pi$ is a nondegenerate nonzero spatial representation $L_{Q}\to \caL(\ell^p(\N))$ which is not injective. So assume that $Q^0$ is the only nonempty saturated and hereditary set of vertices, or equivalently, by 
\cite{libro_Pere}*{Lemma 2.9.6}, that $Q$ is cofinal in the sense of \cite{libro_Pere}*{Definitions 2.9.4} and that it has a cycle $c$ without exits. Cofinality implies that $c$ is the only cycle of $Q$ modulo cycle rotation (by \cite{libro_Pere}*{Lemma 2.7.1 and Theorem 2.7.3}), and that $\sink(Q)=\emptyset$ (by \cite{libro_Pere}*{Lemma 2.9.5}). Moreover, $Q$ cannot have any infinite emmitters. For suppose
$v\in \infem(Q)$; then $v$ cannot be in any cycle, since any cycle containing $v$ would have exits. In particular if $e\in Q^1$ and $s(e)=v$ then $r(e)\ne v$
and by \cite{libro_Pere}*{Lemma 2.0.7} the hereditary and saturated closure of $\{r(e)\}$ does not contain $v$, a contradiction. Hence $Q=\reg(Q)$, and therefore the space $\fX$ of \eqref{fX} consists of the infinite paths of $Q$. If $s(c)=w$, then any such path is of the form $\alpha c^\infty$ for some finite path $\alpha\in\cP$ with
$r(\alpha)=w$. In particular $\fX$ is countable and $\fX_w=\fX_{c^n}=\{c^\infty\}$ for all $n\ge 1$. Hence for the counting measure $\mu$, there is a spatial isometric isomorphism $L^p(\fX,\mu)\cong\ell^p(\N)$, and the nondegenerate representation $\rho_\mu$ of Example \ref{ex:tight} maps $c-c^2$ to zero, so it is not injective. This concludes the proof.
\end{proof}

\begin{rem} By \cite{gardethiel1}, an $L^p$-operator algebra may admit Banach algebra quotients which are not again $L^p$-operator algebras. Thus Phillips' theorem that the $L^p$-Cuntz algebra $\cO^p_d$ is simple as a Banach algebra for $2\le d<\infty$ (\cite{chris3}*{Theorem 5.14}) does not follow from Theorem \ref{thm:simple} above. 
\end{rem}

\section{\(\cO^p(Q)\) vs. \(\cO^{p'}(Q)\)}\label{sec:opoq}
Let $\cR_n$ be the countable graph with exactly one vertex and $n$ loops, $1\le n\le \infty$. We write $L_n=L(\cR_n)$, $\cO^p_n=\cO^p(\cR_n)$. 
In particular, 
\[
L_\infty=\C\{x_i,x_i^*:1\le i\}/\langle x^*_ix_j-\delta_{i,j}\rangle.
\]
\begin{lem}\label{lem:linfty}
Let $Q$ be a countable graph and let $1\le p<\infty$. Assume that $L_Q$ is purely infinite simple. Then there is a homomorphism $L_\infty\to L_Q$ which induces an isometry $\cO^p_\infty\to \cO^p(Q)$. 
\end{lem}
\begin{proof}
Let $\alpha$ be a cycle in $Q$ and let $v=s(\alpha)$. Choose a closed path $\beta$ with $s(\beta)=v$ so that $\alpha$ and $\beta$ are not comparable under the preorder of paths, as in the proof of Lemma \ref{lem:disj}. Then $\beta^*\alpha=\alpha^*\beta=0$ and, of course, $\alpha^*\alpha=\beta^*\beta=v$. Hence there is a $*$-homomorphism $\phi:L_\infty\to L_Q$ such that $\phi(x_i)=\beta^i\alpha$. Observe that if $\rho:L_Q\to\caL(L^p(X))$ is any spatial representation, then $\rho\phi$ is again spatial. Hence $\phi$ induces a contractive homomorphism $\hat{\phi}:\cO^p_\infty\to\cO^p(Q)$. By Theorem \ref{thm:indep}, if $\rho:L_Q\to\caL(L^p(X))$ is a nonzero spatial representation, then $\hat{\phi}$ agrees, up to isometric isomorphism, with the isometric inclusion $\overline{\rho\phi(L_\infty)}\subset\overline{\rho(L_Q)}$.
\end{proof}
\begin{thm}\label{thm:opoq}
Let $Q,Q'$ be countable graphs and let $1\le p\ne p'<\infty$ . Assume that $L_Q$ is purely infinite simple. If in addition, any of the following conditions holds, then there is no nonzero continuous homomorphism $\cO^p(Q)\to \cO^{p'}(Q')$. 
\item[i)] $L_{Q'}$ is simple. 
\item[ii)] $p'\le 2$  and $p\notin (p',2]$.
\item[iii)] $p'>2\ne p$. 
\end{thm}
\begin{proof} Assume there is a nonzero continuous homomorphism $f:\cO^p(Q)\to \cO^{p'}(Q')$. Because the inclusion $L_Q\subset\cO^p(Q)$ is dense, $f(L_Q)\ne 0$, which in view of the simplicity of $L_Q$ implies that $f$ is injective on $L_Q$. Let $\phi:L_\infty\to L_Q$ be as in Lemma \ref{lem:linfty}. Then $f\phi$ is injective, whence $f\hat{\phi}:\cO^p_\infty\to \cO^{p'}(Q')$ is a nonzero continuous homomorphism. Hence there exists $X\in\{\N,[0,1]\}$ and a spatial representation $\rho':L_{Q'}\to \caL(L^{p'}(X))$ such that $\hat{\rho'}f\hat{\phi}:\cO^p_\infty\to\caL(L^{p'}(X))$ is nonzero. By \cite{chris1}*{Lemma 9.1} this implies that $L^{p'}(X)$ contains a subspace isomorphic to $\ell^p(\N)$. If $X=\N$, this cannot be, as noted in the proof of \cite{chris1}*{Theorem 9.2}, by \cite{lindtzaf}*{page 54}; if $X=[0,1]$ and either ii) or iii) holds, this cannot happen either, by \cite{albikalton}*{Theorem 6.4.19}. Thus parts ii) and iii) of the theorem are proved. Part i) also follows, using Proposition \ref{prop:sigmarep} and Theorem \ref{thm:indep}. 
\end{proof}

\end{document}